\documentclass[onefignum,onetabnum]{siamonline250211}

\usepackage{enumitem}
\setlist[enumerate]{leftmargin=.5in}
\setlist[itemize]{leftmargin=.5in}

\usepackage{amsmath,amssymb,mathrsfs,hyperref}
\usepackage{dsfont}
\usepackage{xcolor}
\usepackage{graphicx}
\usepackage{cite}
\usepackage{booktabs}
\usepackage{algorithm}

\newcommand{\change}[1]{{#1}}

\newcommand{\Real}{\mathbb{R}}
\newcommand{\eps}{\varepsilon}
\newcommand{\1}{\mathds{1}}

\renewcommand{\t}{\text{t}}

\renewcommand{\d}{\mathrm{d}}
\newcommand{\grad}{\nabla}
\newcommand{\lap}{\Delta}

\renewcommand{\P}{\mathds{P}}
\newcommand{\Q}{\mathds{Q}}
\newcommand{\E}{\mathds{E}}
\newcommand{\EP}{\E_\P}
\newcommand{\EQ}{\E_\Q}
\newcommand{\F}{\mathcal{F}}
 
\renewcommand{\L}{L}
\DeclareMathOperator{\cov}{cov}
\DeclareMathOperator{\var}{var}
\newcommand{\normal}{\mathrm{N}}

\newcommand{\dkl}{D_{\mathrm{KL}}}
\newcommand{\znu}{\zeta}
\newcommand{\zmu}{\eta}

\newcommand{\zed}{M}

\newsiamremark{remark}{Remark}
\newsiamremark{assumption}{Assumption}

\newcommand{\D}{\mathscr{T}}

\newcommand{\tub}{\mathscr{D}}

\title{
  Relative Entropy Methods for the Approximation of Reactive Trajectories 
  \thanks{BvK was supported by NSF DMS-2012207}}
\author{
  Gabriel Earle \thanks{Department of Mathematics and Statistics, University of Massachusetts, Amherst}
  \and Brian Van Koten \thanks{Department of Mathematics and Statistics, University of Massachusetts, Amherst (\email{bvankoten@umass.edu})}}

\begin{document}
\maketitle

\begin{abstract}  
  Motivated by challenges arising in molecular simulation, we study reactive trajectories of the overdamped Langevin dynamics, i.e.\@ trajectories observed as they pass from a set $A$ corresponding to the reagents of a chemical reaction to a set $B$ corresponding to the products. Reactive trajectories are known to have the same distribution as trajectories of the overdamped Langevin dynamics biased by a singular drift related to the committor function.  In this work, we assess the effect of replacing the exact singular drift with an approximation based on an approximate committor function. We derive a convenient formula for the relative entropy between the distributions of exact and approximate reactive trajectories, and we propose a stochastic gradient descent method for minimizing the entropy to train an approximate committor function on the fly while computing reactive trajectories. We also devise a model assessment procedure for comparing the qualities of different approximations to the committor function based on the relative entropy.
\end{abstract}

\section{Introduction}

For many molecular systems, the most interesting phenomena relate to extremely rare transitions between long-lived states, e.g.\@ transitions between folded states of a protein. Such systems are called \emph{metastable}. It may be infeasible to perform a simulation long enough to observe even one rare transition in a metastable system. Therefore, a wide variety of methods have been devised to compute statistics of rare transitions by means other than direct simulation. Some of these methods, like adaptive multilevel splitting~\cite{cerou_adaptive_2007,cerou_adaptive_2019}, forward flux sampling~\cite{allen_forward_2009}, and transition path sampling~\cite{bolhuis_transition_2002} produce samples of transition paths or fragments of transition paths. Others produce a representative path, for example the minimizer of the Wentzell--Friedlin action, which is roughly speaking the mode of the distribution of transition paths in the limit of low temperature~\cite{vanden-eijnden_geometric_2008}; see~\cite{e_transition-path_2010} for a survey. Finally, one can compute many important statistics of transition paths in terms of solutions of partial differential equations using transition path theory~\cite{e_towards_2006,e_transition-path_2010}. By various means, these methods efficiently explore regions associated with rare transitions. 

Motivated by the challenge of simulating metastable systems, we study \emph{reactive trajectories} of the overdamped Langevin dynamics
\begin{equation}\label{eq: intro overdamped Langevin}
  \d X_t = - \grad U (X_t) \, \d t + \sqrt{2 \eps} \,  \d B_t,
\end{equation}
which are trajectories observed as they pass from a set $A$ corresponding to the reagents of a chemical reaction to a set $B$ corresponding to the products; cf.\@ Figures~\ref{fig: tpp-schematic} and~\ref{fig: reactive-trajectory}. In~\cite{lu_reactive_2015}, building on the transition path theory of E and Vanden-Eijnden~\cite{e_towards_2006}, Lu and Nolen showed that reactive trajectories have the same distribution as the unique strong solution of the \emph{transition path equation}
\begin{equation}\label{eq: intro tpp}
  \d Y_t = - \grad U (Y_t) \, \d t + 2 \eps \grad \log q(Y_t) \, \d t + \sqrt{2 \eps} \,  \d B_t 
\end{equation}
with initial condition
\begin{equation}\label{eq: intro reactive flux distribution}
  Y_0 \sim \frac{1}{\znu} \lvert \grad q(x) \rvert \exp(-U(x)/\eps) \, \d S_A(x).
\end{equation}
Here, 
\begin{equation*}
  q(x) = \P [\tau_B < \tau_A \vert Y_0=x]
\end{equation*}
is the \emph{committor function}, i.e.\@ the probability that a trajectory of the overdamped Langevin dynamics~\eqref{eq: intro overdamped Langevin} with $Y_0=x$ will hit $B$ before $A$, $S_A$ is the surface measure on $\partial A$, and $\znu$ is a normalizing constant. The drift term $2 \eps \grad \log q(X_t)$ in~\eqref{eq: intro tpp} is singular on $\partial A$, since $q$ is zero on $\partial A$, and it forces trajectories away from $A$ and towards $B$ so that the lengths of reactive trajectories will typically be much shorter than the intervals of time between reactions; cf.\@ Figure~\ref{fig: reactive-trajectory}. That is, the singular drift is a bias that forces rare transitions to occur quickly and according to the correct distribution.

We propose a method of training an approximate committor function and singular drift on the fly while computing reactive trajectories. To provide a theoretical foundation for our method, we first assess the errors associated with replacing the exact committor $q$ in the transition path equation with an approximation $\tilde q$. Our main result is a change of measure between the exact distribution of reactive trajectories $\Q$ and the distribution $\P_{\tilde q}$ of~\eqref{eq: intro tpp} with $\tilde q$ in place of $q$. As a corollary, we derive conditions on $\tilde q$ that guarantee the existence of a weak solution of~\eqref{eq: intro tpp} with $\tilde q$ in place of $q$. Note that the standard existence theory for stochastic differential equations does not apply since the drift $2 \eps \grad \log \tilde q$ is not Lipschitz when $\tilde q = 0$ on $\partial A$.\footnote{Note also that many of the arguments in~\cite{lu_reactive_2015} use that $q(Y_t)$ is a martingale and so apply only to the exact process.} 
We derive a formula for the change of measure that can be computed in practice without knowledge of the exact committor or any statistics of the exact transition path process, at least up to a normalizing constant. Considering Girsanov's theorem, one would expect the change of measure to depend on the difference in drifts $2\eps \grad \log q(Y_t) - 2 \eps \grad \log \tilde q (Y_t)$. Surprisingly, we show that one can entirely eliminate the dependence on $q$. Based on this formula, we propose a stochastic gradient descent method for  minimizing the relative entropy $\dkl ( \P_{\tilde q} \Vert \Q)$ to train an approximate committor $\tilde q$ while simultaneously computing approximate reactive trajectories biased by $2 \eps \grad \log \tilde q$.

To explain why we are interested in methods of this type, we wish to make two comparisons with other classes of methods in the computational chemistry literature. First, adaptive multilevel splitting~\cite{cerou_adaptive_2007} computes a sample from the exact distribution of reactive trajectories initiated from a fixed starting point near the boundary of $A$ by branching and killing trajectory fragments. Our methods produce approximate reactive trajectories, but since we calculate the committor and trajectories together, we can ensure that the initial points have at least approximately the reactive flux distribution~\eqref{eq: intro reactive flux distribution}.\footnote{We note that, in~\cite{lopes_analysis_2019}, Lopes and Leli\`evre propose a very promising means of determining the right distribution of initial points for adaptive multilevel splitting, but their method requires the choice of a surface close to $A$ having certain properties that might be difficult to verify in practice. Our approach does not require the identification of any such surface.} Moreover, unlike adaptive multilevel splitting, our approach does not require the explicit specification of a reaction coordinate. 
In fact, while computing reactive trajectories, we solve for the committor, which is in some respects the ideal reaction coordinate~\cite{cerou_adaptive_2019}.

Second, one can compute the committor as the solution of the Kolmogorov equation~\eqref{eq: elliptic commitor equation} using numerical methods for high-dimensional partial differential equations such as physics-informed neural networks (PINNs)~\cite{khoo_solving_2019,li_computing_2019}, tensor networks~\cite{chen_committor_2023}, or kernel methods~\cite{evans_computing_2022,aristoff_fast_2024}. To calculate statistics of reactive trajectories given the committor, for example the crossover time or reaction rate, one can take averages over the Boltzmann distribution; cf.\@~\eqref{eq: tpt formula for crossover time}. We note, however, that summary statistics like the reaction rate do not always adequately characterize the reactive trajectories, especially when a reaction may occur by multiple mechanisms. Therefore, even given an accurate estimate of the committor, one might wish to sample reactive trajectories.\footnote{See also~\cite{yuan_optimal_2024} for evidence that, under some circumstances, rates computed directly from trajectories biased by an approximate singular drift might be more accurate than rates computed via formulas from transition path theory as in~\eqref{eq: tpt formula for crossover time}.}  Moreover, it may be very difficult to assess the accuracy of an approximate committor when the state space is high-dimensional. The usual \emph{a posteriori} error estimates for numerical solutions of partial differential equations on low-dimensional spaces do not apply. Our method provides an estimate of the relative entropy $\dkl (\P_{\tilde q} \Vert \Q)$, which directly measures the error in the distribution of reactive trajectories. Finally, in the implementation of methods like PINNs, one must choose a set of collocation points. For computing committor functions of metastable systems, one may have to use enhanced sampling methods like stratification or importance sampling to generate a suitable set of collocation points covering the transition region~\cite{rotskoff_active_nodate}. We propose computing trajectories biased by an approximate singular drift as a convenient means to similar ends.

To summarize, our proposed relative entropy optimization method computes both reactive trajectories and the committor, with error estimates, with at least approximately the correct distribution of initial points, and without requiring an explicit choice of collocation points or a reaction coordinate. Instead, the user must specify an initial approximation $\tilde q_0$ to the committor function. We discuss the choice of $\tilde q_0$ in Section~\ref{sec: numerical experiments}. Here, we note only that one could choose $\tilde q_0$ so that the transition path process with $\tilde q_0$ in place of $q$ imposes a biasing force pushing on a small number of atoms to induce transitions to occur as in steered molecular dynamics~\cite{isralewitz_steered_2001}. 

In addition to optimizing the relative entropy, we demonstrate two other potential applications of our theoretical results: The assessment of models of the committor and importance sampling. For high-dimensional systems, the committor and related quantities such as the density and current of transition paths may be difficult to interpret without coarse-graining, i.e.\@ dimensionality reduction. We suggest that one could select or optimize coarse models of the committor based on the relative entropy; cf.\@ Section~\ref{sec: numerical experiments}. We also note that one can use importance sampling to correct for errors in the distribution of reactive trajectories generated based on an approximate committor. However, the change of measure involves an exponential average, and experience with similar importance sampling methods based on Girsanov's theorem suggests that such an approach would not be reliable for poor approximations of the committor; cf.\@ Section~\ref{sec: importance sampling}. Nonetheless, importance sampling yields good results for our simple test problem in Section~\ref{sec: crossover times}. 

We warn the reader that all theoretical and computational applications of our results demand that care be taken to correctly handle singularities. For example, since the drift is singular standard numerical integrators may fail. We have proposed an operator splitting scheme for integrating~\eqref{eq: intro tpp} in the case where $\partial A$ is planar. We do not prove convergence of the integrator, nor do we propose an integrator that works when $\partial A$ is curved, although we do grant that one would often want a curved boundary. We also note that the change of measure formula~\eqref{eq: computable change of measure} involves an integral whose integrand may be singular; cf.\@ Section~\ref{sec: alternative expression}.

  Finally, we note that we are only able to prove the existence of the gradient $\grad_\theta \dkl ( \P_{q_\theta} \Vert \Q)$ under rather stringent conditions; cf.\@ Assumption~\ref{asm: properties that guarantee differentiability of dkl}. We suspect, but cannot prove, that the relative entropy is differentiable under weaker conditions. In any case, our method of training the committor by stochastic gradient descent appears to work well even when Assumption~\ref{asm: properties that guarantee differentiability of dkl} fails; cf.\@ Section~\ref{subsec: training improved approximate committor}. Note that we require Assumption~\ref{asm: properties that guarantee differentiability of dkl} only to prove differentiability, not to derive the change of measure formula in Theorem~\ref{thm: complete change of measure from transition path process to simulated process} or to prove the weak law of large numbers for the estimator of relative entropy differences in Section~\ref{sec: selection}.
\\~\\
\noindent \emph{Outline.} In Section~\ref{sec: transition path theory}, we briefly review transition path theory. In Section~\ref{sec: change of measure formula}, we derive the change of measure between the exact distribution of transition paths and the distribution of~\eqref{eq: intro tpp} with an approximate committor $\tilde q$ in place of $q$. In Section~\ref{sec: applications}, we explain how to estimate relative entropy differences and gradients, and we consider the possibility of importance sampling. In Section~\ref{sec: numerical methods}, we address some practical obstacles to the implementation of our methods, including strategies for the representation of approximate committor functions that satisfy Assumption~\ref{asm: properties of approximate committor} and Assumption~\ref{asm: properties that guarantee differentiability of dkl}. Finally, in Section~\ref{sec: numerical experiments}, we apply our methods to a two-dimensional toy model. First, we propose an integrator capable of handling the singular drift. We then assess the quality of a coarse approximation of a committor, train an improved approximation by stochastic gradient descent, demonstrate the use of importance sampling to correct for errors in the committor, and study the convergence of the integrator.

\section{A Digest of Transition Path Theory}
\label{sec: transition path theory}

In this section, we present a brief review of \emph{transition path theory}~\cite{e_towards_2006,lu_reactive_2015}. Let $U: \Real^d \rightarrow \Real$ be a potential energy, describing for example a molecular system. Let $\eps = k_B T >0$ be a temperature parameter, and suppose that $X_t$ evolves according to the overdamped Langevin dynamics
\begin{equation}\label{eq: overdamped Langevin}
\d X_t = - \grad U (X_t) \, \d t + \sqrt{2 \eps} \,  \d B_t.
\end{equation}
Under certain conditions on $U$~\cite{lelievre_partial_2016}, the overdamped Langevin dynamics is ergodic for the Boltzmann distribution
\begin{equation*}
  \rho(\d x ) = Z^{-1} \exp(- U(x)/ \eps) \, \d x \text{ where } Z = \int_{\Real^d}  \exp(- U(x)/ \eps) \, \d x.
\end{equation*}
We assume ergodicity throughout this work. 

\begin{assumption}\label{asm: ergodicity}
For the given potential function $U$, the overdamped Langevin dynamics~\eqref{eq: overdamped Langevin} is ergodic. 
\end{assumption}

In transition path theory, to characterize rare transitions of the overdamped Langevin dynamics, one first chooses a disjoint pair of subsets $A$ and $B$ of $\Real^d$. In applications, $A$ and $B$ would usually be associated with metastable states, e.g.\@ they might be sets modeling different conformations of a biomolecule or the reagents and products of a chemical reaction. In addition to the sets $A$ and $B$, we find it convenient to define the \emph{transition region}
\begin{equation*}
\D = \Real^d \setminus (A \cup B).
\end{equation*}
We impose the following assumptions on $A$, $B$, and $\D$.

\begin{assumption}\label{asm: submanifold assumption}
  We assume that $A$ and $B$ are disjoint, closed subsets of $\Real^d$ with nonempty interior. The boundaries of $A$ and $B$ are regular $C^\infty$-submanifolds of $\Real^d$. The transition region $\D$ is connected. 
\end{assumption}

\begin{figure}
  \caption{Schematic illustration of reactive trajectories, entrance times, and exit times. The orange line depicts a trajectory of the overdamped Langevin dynamics. The bold red segments depict reactive trajectories.  Here, $\tau_{B,0}$ is the first hitting time of $B$, $\tau_{A,1}$ is the first hitting time of $A$ after $\tau_{B,0}$, $\tau_{B,1}$ is the first hitting time of $B$ after $\tau_{A,1}$, and so forth. The $k$'th exit time $\sigma_{A,k}$ is the last time before $\tau_{B,k}$ that the process was in $A$.}
  \begin{center}
    \includegraphics[width=0.5\linewidth]{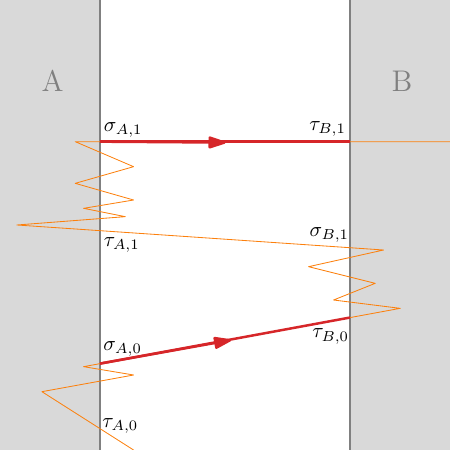}
  \end{center}
  \label{fig: tpp-schematic}
\end{figure}

\begin{figure}
  \caption{Reactive trajectories. On the left, we have a single long trajectory of the overdamped Langevin dynamics for the simple two-dimensional model potential introduced in Section~\ref{sec: numerical experiments}. The blue curves are contours of the potential. The trajectory was initialized from the Boltzmann distribution. Its starting point happened to be in the set $A$ on the left side of the figure. The trajectory was terminated on hitting the boundary of $B$ for the first time. The bold red curve depicts the end of the trajectory from the last time it left $A$ to the first time it entered $B$, i.e.\@ the reactive part of the trajectory. The faint orange curve depicts the beginning of the trajectory up to the last time it left $A$. On the right, we have three approximate reactive trajectories computed by integrating the transition path process for the approximate committor $q_{\rm{sgd}}$ calculated by gradient descent in Section~\ref{subsec: training improved approximate committor}. The blue curves are contours of the effective potential $U - 2\eps \log q_{\rm{sgd}}$ for the transition path process, which is singular on $\partial A$. The red and pink curves are the reactive trajectories.}
  \begin{center}
    \includegraphics{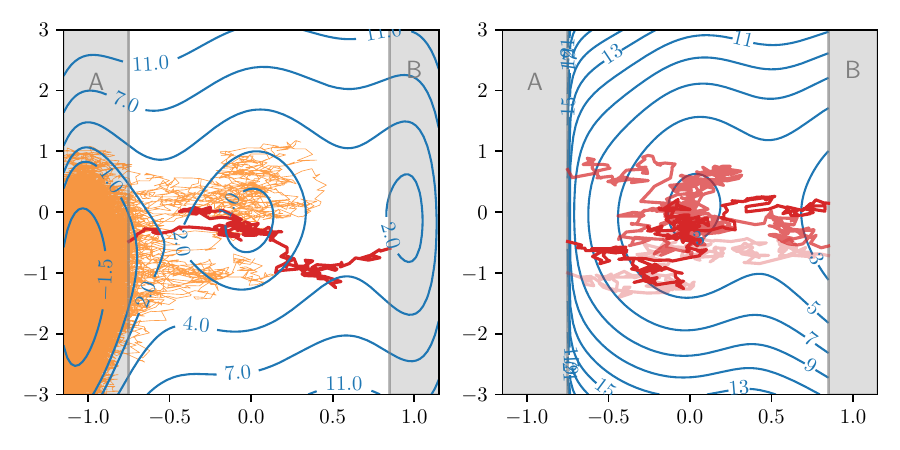}
    \end{center}
  \label{fig: reactive-trajectory}
\end{figure}

Given a pair of sets $A$ and $B$, one defines sequences of \emph{entrance times}, \emph{exit times}, and \emph{reactive trajectories} as illustrated in Figure~\ref{fig: tpp-schematic}.
The entrance times are
\begin{align*}
  \tau_{A,0} &:= \inf \{ t \geq 0; X_t \in A\}, \\
  \tau_{B,0} &:= \inf \{ t \geq \tau_{A,0}; X_t \in B\}, 
\end{align*}
and for $k \geq 1$
\begin{align*}
  \tau_{A,k} &:= \inf \{ t \geq \tau_{B,k-1}; X_t \in A\}, \\
  \tau_{B,k} &:= \inf \{ t \geq \tau_{A,k-1}; X_t \in B\}.
\end{align*}
The exit times are
\begin{align*}
  \sigma_{A,0} &:= \sup \{ \tau_{A_0} \leq t \leq \tau_{B,0}; X_t \in A\}, \\
  \sigma_{B,0} &:= \sup \{\tau_{B,0} \leq  t \leq \tau_{A,1}; X_t \in B\}, 
\end{align*}
and for $k \geq 1$
\begin{align*}
  \sigma_{A,k} &:= \sup \{ \tau_{A,k} \leq t \leq \tau_{B,k}; X_t \in A\}, \\
  \sigma_{B,k} &:= \sup \{\tau_{B,0} \leq  t \leq \tau_{A,1}; X_t \in B\}.
\end{align*}
The $k$'th reactive trajectory (or transition path) $Y^k_t$ is the segment of the random path $X_t$ between the $k$'th exit time $\sigma_{A,k}$ from $A$ and the $k$'th entrance time $\tau_{B,k}$ to $B$. That is,
\begin{equation*}
Y^k_t := X_{\sigma_{A,k} + t} \text{ for } t \in [0, \tau_{B,k} - \sigma_{A,k}]. 
\end{equation*}

Transition path theory characterizes the distribution of reactive trajectories in terms of the \emph{committor function} 
\begin{equation*}
  q(x) := \P[ \tau_B < \tau_A \vert X_0 = x].
\end{equation*}
Here, $\tau_A$ and $\tau_B$ are the first hitting times of $A$ and $B$ for $X_t$, so $q(x)$ is the probability of hitting $B$ before $A$ when starting from $x$. 
In~\cite{lu_reactive_2015}, Lu and Nolen verified that if the overdamped Langevin dynamics is in the steady state with $X_t \sim \rho$ for all $t \geq 0$, then the reactive trajectories have the same distribution as the solution of the \emph{transition path equation}
\begin{equation}\label{eq: tpp}
    \d Y_t = -\grad U (Y_t) \, \d t + 2 \eps \grad \log q(Y_t) \, \d t + \sqrt{2 \eps} \, \d B_t
\end{equation}
where $Y_0$ has the \emph{reactive flux distribution}
\begin{equation}\label{eq: reactive flux distribution}
  Y_0 \sim \frac{1}{\znu} \lvert  \grad q(x) \rvert \exp(-U(x)/\eps) \, \d S(x). 
\end{equation}
Here, $S$ is the surface measure on $\partial A$, and
\begin{equation*}
  \znu = \int_{\partial A} \lvert  \grad q(x) \rvert \exp(-U(x)/\eps) \, \d S(x).
\end{equation*}

Observe that the transition path equation is the overdamped Langevin dynamics with the additional drift term $2 \eps \grad \log q(X_t) \, \d t$. On $\partial A$, we have $q = 0$, and by Lemma~\ref{lem: properties of committor} below, $\lvert \grad q \rvert >0$. Therefore, the drift $\grad \log q = \frac{\grad q}{q}$ is singular on $\partial A$. Despite the singularity, Lu and Nolen showed that a unique strong solution of~\eqref{eq: tpp} exists for any initial distribution, even distributions supported on $\partial A$. Moreover, with probability one, $X_t \notin A$ for $t >0$. In effect, the singular drift repels trajectories from $A$, forcing transitions to occur. 

\begin{remark}
One can derive~\eqref{eq: tpp} formally by conditioning the overdamped Langevin dynamics on hitting $B$ before $A$ using the Doob $h$-transform. This approach does not establish a definite connection with reactive trajectories or the existence of strong solutions. 
\end{remark}

The committor solves the elliptic boundary value problem
\begin{equation}\label{eq: elliptic commitor equation}
  \begin{cases}
    L q = 0 &\text{ in } \D, \\
    q=0 &\text{ on } \partial A, \\
    q = 1 &\text{ on } \partial B,\
  \end{cases}
\end{equation}
where
\begin{equation*}
L := - \grad U \cdot \grad + \eps \lap
\end{equation*}
is the generator of the overdamped Langevin dynamics~\cite{e_towards_2006}.  Therefore, under our smoothness assumptions on $U$ and $\D$, elliptic regularity, the Hopf lemma, and the maximum principle imply the properties of the committor listed in Lemma~\ref{lem: properties of committor} below. 

\begin{assumption}\label{asm: smoothness of potential}
We assume that the potential energy $U: \Real^d \rightarrow \Real$ is infinitely differentiable, but we do not assume that $U$ or any of its derivatives are bounded. 
\end{assumption}

\begin{lemma}
  \label{lem: properties of committor}
  Under Assumptions~\ref{asm: submanifold assumption} and~\ref{asm: smoothness of potential}, the forward committor function $q: \D \rightarrow [0,1]$ extends to an infinitely differentiable function defined on an open set containing $\partial A$ and $\partial B$. By abuse of notation, we let $q$ refer to both the committor function and the extension.
  For all $x \in \mathring{\D}$,
  \begin{equation*}
    q(x) >0.
  \end{equation*}
  For all $x \in \partial A$,
  \begin{equation*}
    \grad q(x)\cdot n(x) = \lvert \grad q(x) \rvert > 0,
  \end{equation*}
  where $n(x)$ is the outward unit normal to $A$ at $x$. 
\end{lemma}

\begin{proof}
  By elliptic regularity, there exists an infinitely differentiable extension of $q$ to an open set containing $\overline{\D}$. We have $\grad q(x)\cdot n(x) >0$ for $x \in \partial A$ by the Hopf Lemma, and $q(x) >0$ in the interior of $\D$ by the maximum principle. Moreover, since $q$ is constant on $\partial A$ and $\grad q(x) \cdot n(x) >0$, $\grad q(x) = \lvert \grad q(x) \rvert n(x)$.  
\end{proof}

\begin{remark}\label{rem: smoothness of q}
  It is conventional to take the domain of the committor to be $\Real^d$ instead of $\D$, letting $q=0$ on $A$ and $q=1$ on $B$. \emph{This extension of $q$ does not coincide with the smooth extension described in Lemma~\ref{lem: properties of committor}}. Any smooth extension must take negative values at some points in the interior of $A$ since $\grad q \cdot n >0$ and $q =0$ on $\partial A$. Moreover, observe that for our smooth extension, $Lq=0$ everywhere in $\bar \D$, including on $\partial A$. In the original work on transition state theory, $Lq$ was instead understood as a distribution supported on $\partial A \cup \partial B$. We require a smooth extension only to simplify the notation in certain proofs. None of our results depends on a particular choice of extension. 
\end{remark}

\section{A Change of Measure Formula for Approximations to the Transition Path Process}
\label{sec: change of measure formula}

We analyze the errors that result when one substitutes an approximation $\tilde q$ of the committor in the transition path equation~\eqref{eq: tpp}. To be precise, in Theorem~\ref{thm: complete change of measure from transition path process to simulated process} below, we present a change of measure formula relating the distribution of solutions of the approximate transition path equation 
\begin{equation}
  \label{eq: approximate tpp}
  \d Y_t = -\grad U (Y_t) \, \d t + 2 \eps \grad \log \tilde q(Y_t) \, \d t + \sqrt{2 \eps}\, \d B_t
\end{equation}
with the exact transition path process. To derive the change of measure, motivated by Lemma~\ref{lem: properties of committor}, we impose the following assumptions on the approximate committor $\tilde q$.

\begin{assumption}\label{asm: properties of approximate committor}
  Let $\tilde q : \D \rightarrow \Real$ be an approximation to the committor with the following properties:
  \begin{enumerate}
  \item $\tilde q$ extends to an infinitely differentiable function defined on an open set containing $\partial A$ and $\partial B$. 
  \item $\tilde q(x) = 0$ for all $x \in \partial A$.
    \item $\grad \tilde q(x) \cdot n(x) >0$ for all $x \in \partial A$ where $n(x)$ is the outward unit normal to $A$ at $x$.
  \end{enumerate}
\end{assumption}

We now show that when an approximate committor $\tilde q$ has the properties outlined above, the ratio $\frac{\tilde q}{q}$ must be smooth in a neighborhood of $\partial A$. This will be the crucial property in deriving sufficient conditions for a change of measure.  

\begin{lemma} \label{lem: ratio function is smooth}
  Let $q$ be the forward committor. Let $\tilde q$ be an approximation to $q$ having the properties outlined in Assumption~\ref{asm: properties of approximate committor}.
  The function
  \begin{equation*}
    r(x) = \frac{\tilde q(x)}{q(x)}
  \end{equation*}
  defined for $x \in \mathring{\D}$ extends to a function in $C^\infty(\overline{\D})$. By abuse of notation, we let $r$ denote both the function and its extension. For $x \in \partial A$, we have
  \begin{equation*}
     r(x) = \frac{\lvert \grad \tilde q(x) \rvert}{\lvert \grad q(x) \rvert} = \frac{\grad \tilde q(x) \cdot n(x)}{\grad q(x) \cdot n(x)}.
  \end{equation*}
\end{lemma}

\begin{proof}
  Under Assumption~\ref{asm: submanifold assumption}, there exists a tubular neighborhood of $\partial A$. That is, there exist an open set $\tub$ containing $\partial A$, an infinitely differentiable retraction $\rho : \tub \rightarrow \partial A$, and an infinitely differentiable mapping $h: \tub \rightarrow \Real$ so that for any $x \in \tub$, 
\begin{equation*}
  x = \rho(x) + h(x) n(\rho(x))
\end{equation*}
where $n(y)$ is the outward unit normal to $A$ at $y \in \partial A$.  To simplify notation, we define
\begin{equation*}
  \bar x := \rho(x),
\end{equation*}
and we sometimes write $h$ and $n$ for $h(x)$ and $n(\rho(x))$. We choose $\tub$ so that $\tub$ and $B$ are disjoint. 

  Since $\tilde q >0$ on $\D \cup \partial B$, $r$ is smooth in an open neighborhood of any point in $\D \cup \partial B$. It will suffice to show that $r$ is smooth up to $\partial A$. 
  Let $x \in \tub$. 
  By Taylor's theorem and Lemma~\ref{lem: properties of committor},
  \begin{align*}
    q(x) &= q(\bar x + h n) \\
         &= q(\bar x) + (\grad q(\bar x) \cdot n)h + h^2 \int_{s=0}^1 n^t D^2 q(\bar x + shn) n (1-s) \, \d s \\
    &= \lvert \grad q(\bar x) \rvert h +  h^2 \int_{s=0}^1 n^t D^2 q(\bar x + shn) n (1-s) \, \d s.
  \end{align*}
  Similarly, we have
  \begin{equation*}
    \tilde q(x) = \lvert \grad \tilde q(\bar x) \rvert h +  h^2 \int_{s=0}^1 n^t D^2 \tilde q(\bar x + shn) n (1-s) \, \d s,
  \end{equation*}
  since the proof of Lemma~\ref{lem: properties of committor} verifies that $\grad \tilde q(\bar x) = \lvert \grad \tilde q(\bar x) \rvert n$ under our assumptions on $\tilde q$.
   Therefore,
  \begin{equation*}
    \frac{\tilde q( x)}{q( x)}
    =
    \frac{\lvert \grad \tilde q(\bar x) \rvert}{\lvert \grad q(\bar x) \rvert}
    \frac{1+ \frac{h}{\lvert \grad \tilde q(\bar x) \rvert}  \int_{s=0}^1 n^t D^2 \tilde q(\bar x + shn) n (1-s) \, \d s }{1+\frac{h}{\lvert \grad q(\bar x) \rvert}  \int_{s=0}^1 n^t D^2  q(\bar x + shn) n (1-s) \, \d s}.
  \end{equation*}
  The result follows since $\bar x$ and $h$ are smooth functions of $x$, $\tilde q$ and $q$ are in $C^\infty(\overline{\D})$, and $\lvert \grad \tilde q(\bar x) \rvert$ and $\lvert \grad \tilde q(\bar x) \rvert$ are positive. 
\end{proof}

Lemma~\ref{lem: ratio function is smooth} implies that $\log r$ and $\grad \log r$ are smooth and bounded in a neighborhood of any point on $\partial A$, since $\lvert \grad q (x) \rvert >0$ for all $x \in \partial A$ by Lemma~\ref{lem: properties of committor}. If $\grad \log r$ is in fact globally bounded, then the Novikov condition stated in Assumption~\ref{asm: novikov} below holds, and Girsanov's theorem guarantees a change of measure. 

\begin{assumption}\label{asm: novikov}
   For any $t >0$, we have the Novikov condition
  \begin{equation*}
    \EQ \left [ \exp \left (\eps \int_0^t \lvert \grad \log r(X_s) \rvert^2 \, \d s \right ) \right ] < \infty.
  \end{equation*}
\end{assumption}

Establishing or assuming a global bound on $\grad \log r$ seems to us to be the most practical means of verifying the Novikov condition and the existence of a change of measure. For example, a global bound on $\grad \log r$ follows from Assumption~\ref{asm: properties of approximate committor} and Lemma~\ref{lem: ratio function is smooth} when $\D$ is bounded.  We note, however, that a global bound is only a convenient sufficient condition, not necessary, for the results below. 

% First of all, Novikov's condition is itself only a sufficient condition for a change of measure. Second, Novikov's condition does not imply that $\grad \log r$ is bounded even in a neighborhood of $\partial A$. Therefore, it is not in fact necessary that Assumption~\ref{asm: properties of approximate committor} hold. Assumption~\ref{asm: novikov} is merely a convenient sufficient condition that can be guaranteed in practice by requiring $\tilde q$ to have a certain functional form; cf.\@ (CITE). If one assumes in addition that $\overline{\Omega}$ is compact, then Assumption~\ref{asm: properties of approximate committor} implies a global bound on $\grad \log r$ and therefore Novikov's condition. 

 To state our change of measure, we require some additional notation to properly specify the probability space on which the exact and approximate transition path processes are defined. Let $B_t$ be a $d$-dimensional standard Brownian motion on the probability space $(\Omega, \F, \Q)$, and suppose that $Y_0$ is a random variable on $(\Omega, \F)$ taking values in $\overline{\D}$ and independent of $B_t$. For now, we do not make any assumptions regarding the distribution of $Y_0$, but, in Theorem~\ref{thm: complete change of measure from transition path process to simulated process} below, $Y_0$ will have the reactive flux distribution~\eqref{eq: reactive flux distribution}. Let $\F_t$ be the filtration generated by $B_t$ and $Y_0$. In~\cite{lu_reactive_2015}, it is shown that a strong solution of
\begin{equation}\label{eq: sde for tpp}
  dY_t = -\grad U(Y_t) \, \d t + 2 \eps \grad \log q(Y_t) \, \d t + \sqrt{2\eps} \, \d B_t
\end{equation}
exists with initial value $Y_0$. 
Under our assumptions, Girsanov's theorem yields a change of measure relating the transition path process~\eqref{eq: sde for tpp} to a weak solution of the approximate transition path equation
\begin{equation*}
  \d Y_t = -\grad U (Y_t) \, \d t + 2 \eps \grad \log \tilde q(Y_t) \, \d t + \sqrt{2 \eps} \, \d B_t.
\end{equation*}

\begin{lemma}\label{lem: impractical general change of measure formula}
 Let Assumptions~\ref{asm: submanifold assumption},~\ref{asm: smoothness of potential},~\ref{asm: properties of approximate committor}, and~\ref{asm: novikov} hold. Let $Y_t$ and $Y_0$ be as described in the previous paragraph. The process
  \begin{equation}\label{eqn: first formula for z}
    \zed_t := \exp \left ( \int_0^t \sqrt{2 \eps} \grad \log r(Y_s) \cdot \d B_s - \frac12 \int_0^t \lvert \sqrt{2 \eps} \grad \log r(Y_s) \rvert^2 \, \d s \right )
  \end{equation}
  is a positive $\F_t$-martingale under $\Q$ with $\EQ [\zed_t] = 1$ for all $t\geq 0$. Therefore, there exists a probability measure $\P$ on $(\Omega, \F)$ so that for any fixed $T>0$,
  \begin{equation*}
    \left . \frac{\d \P}{\d \Q} \right \rvert_{\F_T} = \zed_T.
  \end{equation*}
  Define the process $W_t$  by
  \begin{equation*}
     \sqrt{2 \eps} W_t := Y_t - Y_0 - \int_0^t -\grad U(Y_s) + 2 \eps\grad \log \tilde q(Y_s) \, \d s
   \end{equation*}
   so that $Y_t$ solves
   \begin{equation*}
    \d Y_t =  -\grad U(Y_t) \, \d t + 2 \eps \grad \log \tilde q(Y_t) \, \d t + \sqrt{2 \eps} \, \d W_t.
  \end{equation*}
  Under $\P$, $\{W_t\}_{t=0}^T$ has the law of a standard Brownian motion.
\end{lemma}

\begin{proof}
 \change{The above merely restates Girsanov's theorem; cf.\@\cite[Section~8.6]{oksendal_stochastic_2003}}. Observe that Girsanov's theorem holds even when the drift is singular. Only some form of Novikov's condition is required. 
\end{proof}

\begin{remark}[Weak Solutions of the Approximate Transition Path Equation]
Lemma~\ref{lem: impractical general change of measure formula} establishes the existence of at least a weak solution to 
\begin{equation*}
  \d Y_t =  -\grad U(Y_t)\, \d t + 2 \eps \grad \log \tilde q(Y_t) \, \d t + \sqrt{2 \eps} \, \d B_t
\end{equation*}
with arbitrary initial distribution supported on $\partial A$ even though $\grad \log \tilde q$ is singular on $\partial A$. We do not verify the existence of strong solutions. However, we suspect that one could do so based on arguments roughly similar to those used in~\cite{lu_reactive_2015} to construct strong solutions of the transition path equation~\eqref{eq: tpp}.
\end{remark}

We now derive a formula for the change of measure that can be computed in practice. Observe that formula~\eqref{eqn: first formula for z} for $\zed_t$ cannot be computed without knowledge of the exact committor $q$. As a first step towards eliminating this dependence on the committor, we show that $\log q$ solves a Hamilton--Jacobi--Bellman equation. Results of this nature are standard; we include a proof only for the reader's convenience. 

\begin{lemma}\label{lem: hjb equation}
  For $x \in \mathring{\D}$,
  \begin{align*}
    0 &=-\grad U (x) \cdot \grad \log q(x) + \eps \lap \log q(x) + \eps \lvert \grad \log q (x) \rvert^2 \\
      &= \L \log q (x) + \eps \lvert \grad \log q(x) \rvert^2.
  \end{align*}
\end{lemma}

\begin{proof}
   Define 
  \begin{equation*}
    m(x):= \log q(x)
  \end{equation*}
  for $x \in \D$. We have 
  \begin{align*}
    0&= \L q \\
     &= -\grad U \cdot \grad q + \eps \lap q \\
     &= - \grad U \cdot \grad \exp(m ) + \eps \lap \exp(m) \\
     &= \exp(m) ( -\grad U  \cdot \grad m + \eps \lap m + \eps \lvert \grad m\rvert^2 )
  \end{align*}
  on $x \in \mathbb{R}^d \setminus (A \cup B)$, which implies the result since $\exp(m)$ must be positive.
\end{proof}

We use the Hamilton--Jacobi--Bellman equation in Lemma~\ref{lem: hjb equation} to eliminate $q$ from the exponential factor in formula~\eqref{eqn: first formula for z} for $\zed_t$. 

\begin{lemma}\label{lem: second impractical change of measure formula}
  Under the hypotheses of Lemma~\ref{lem: impractical general change of measure formula}, and assuming $Y_0$ takes values in $\partial A$,  we have
  \begin{align}
    \zed_t&= \frac{\lvert \grad  q (Y_0) \rvert}{\lvert   \grad \tilde q (Y_0) \rvert} \frac{\tilde q(Y_t)}{q(Y_t)} \exp \left (-\int_0^t  \L \log \tilde q (Y_s)  + \eps  \lvert \grad \log \tilde q(Y_s) \rvert^2 \, \d s  \right ) \nonumber \\
    &= \frac{\lvert \grad  q (Y_0) \rvert}{\lvert \grad \tilde q (Y_0) \rvert} \frac{\tilde q(Y_t)}{q(Y_t)} \exp \left (-\int_0^t \frac{L \tilde q}{\tilde q}(Y_s) \, \d s  \right ) \label{eqn: second formula for zt}
  \end{align}
  for $t>0$ and $\zed_0=1$. 
\end{lemma}

\begin{proof}
  First, we derive an alternative expression for the integral with respect to $\d B_s$ that appears in formula~\eqref{eqn: first formula for z} for $\zed_t$.  By the Ito formula, for any $g \in C^2(\Real^d ; \Real)$, we have
  \begin{align*}
    \d g(Y_t) &= \grad g(Y_t)^t \d Y_t + \frac12 \d Y_t^t D^2g(Y_t) \d Y_t \\
              &= \L g (Y_t) \, \d t + 2 \eps \grad \log q(Y_t)^t \grad g(Y_t) \, \d t +  \sqrt{2 \eps}\grad g(Y_t)^t \d B_t,
  \end{align*}
  so
  \begin{align*}
    \sqrt{2 \eps} \int_0^t \grad g(Y_s)^t \d B_s = g(Y_t) - g(Y_0) - \int_0^t   \L g (Y_s) + 2 \eps \grad \log q(Y_s)^t \grad g(Y_s) \, \d s.
  \end{align*}
  Applying the above with $g = \log r$ and using Lemma~\ref{lem: ratio function is smooth}, we have for $t >0$ that 
  \begin{align*}
    \zed_t &= \frac{r(Y_t)}{r(Y_0)} \exp \left (-\int_0^t  \L \log r(Y_s) + 2 \eps \grad \log q(Y_s)^t  \grad \log r(Y_s) + \eps \lvert \grad \log r(Y_s) \rvert^2 \, \d s  \right ) \\
    &= \frac{\lvert \grad  q (Y_0) \rvert}{\lvert  \tilde \grad q (Y_0) \rvert} \frac{\tilde q(Y_t)}{q(Y_t)} \\
    &\quad \times
           \exp \left (-\int_0^t  \L \log r(Y_s) + 2 \eps \grad \log q(Y_s)^t  \grad \log r(Y_s) + \eps \lvert \grad \log r(Y_s) \rvert^2 \, \d s  \right ). 
  \end{align*}
  By Lemma~\ref{lem: hjb equation},
  \begin{align*}
    &\L \log r + 2 \eps \grad \log q \cdot \grad \log r + \eps \lvert \grad \log r \rvert^2 \\
    &\qquad = \L \log \tilde q - \L \log q + 2 \eps \grad  \log q \cdot \grad \log \tilde q - 2 \eps \lvert \grad \log q \rvert^2 \\
    &\qquad \qquad+ \eps \lvert \grad \log \tilde q \rvert^2 + \eps \lvert \grad \log q \rvert^2 - 2 \eps \grad  \log q \cdot \grad \log \tilde q \\
    &\qquad = \L \log \tilde q + \frac12 \lvert \grad \log \tilde q \rvert^2.
  \end{align*}
  Thus, 
  \begin{equation*}\label{eq: first formula for zxt}
    \zed_t= \frac{\lvert \grad  q (Y_0) \rvert}{\lvert  \tilde \grad q (Y_0) \rvert} \frac{\tilde q(Y_t)}{q(Y_t)} \exp \left (-\int_0^t  \L \log \tilde q (Y_s)  + \eps \lvert \grad \log \tilde q(Y_s) \rvert^2 \, \d s  \right ).
  \end{equation*}
  Here, for $t>0$, we have $Y_t \in \mathring{\D}$ with probability one, so $q(Y_t) >0$, and the right-hand-side of the equation above is well-defined.  In addition, $\zed_0=1$ by~\eqref{eqn: first formula for z}. This verifies the first formula for $\zed_t$ in the statement of the lemma.
  The second formula follows from the first and the identity
  \begin{align*}
    L \log \tilde q &= \eps \lap \log  \tilde q - \grad U \cdot \grad \log \tilde q \\
    % &= \eps \frac{\lap \tilde q}{\tilde q} - \eps  \lvert \grad \log \tilde q \rvert^2 - \grad U \cdot \grad \log \tilde q \\
    &= \frac{L \tilde q}{\tilde q}  - \eps  \lvert \grad \log \tilde q \rvert^2.
  \end{align*}
\end{proof}

We must now eliminate the factors $\lvert \grad q(Y_0) \rvert$ and $q(Y_t)$ from the change of measure formula~\eqref{eqn: second formula for zt}. To get rid of $q(Y_t)$, we simply stop observing the process at the first hitting time $\tau$ of $B$; note that $Y_\tau \in \partial B$, so $q(Y_\tau)= 1$. Observing the process up to time $\tau$ corresponds to restricting the measures $\P$ and $\Q$ to the stopping time $\sigma$-algebra defined below.

\begin{definition}
  Let $\tau$ be the first hitting time of $B$ for the process $Y_t$ solving~\eqref{eq: sde for tpp}. Let $\F_\tau$ be the stopping time $\sigma$-algebra of $\tau$, i.e.
  \begin{equation*}
    \F_\tau := \{ A \in \F: A \cap \{ \tau \leq t\} \in \F_t \text{ for all } t \geq 0 \}.
  \end{equation*}
\end{definition}

For Theorem~\ref{thm: complete change of measure from transition path process to simulated process} below to hold, $\tau$ must be finite for both the exact and approximate transition path processes. Assumption~\ref{asm: ergodicity} on the ergodicity of the overdamped Langevin dynamics implies that the exact transition path process will hit $B$ with probability one, since $B$ has positive Lebesgue measure. That is, $\Q[\tau < \infty] = 1$. We assume that $\tau$ is finite for the approximate process as well. 

\begin{assumption}\label{asm: finite hitting time}
  We assume that 
  \begin{equation*}
     \P[\tau < \infty] = 1
  \end{equation*}
  for any initial value $Y_0$ supported in $\overline{\D}$. Recall that $\P$ is the measure introduced in Lemma~\ref{lem: impractical general change of measure formula} under which $Y_t$ is a weak solution of the approximate transition path equation~\eqref{eq: approximate tpp}.
\end{assumption}

We eliminate $\lvert \grad q(Y_0)\rvert$ from~\eqref{eqn: second formula for zt} by taking $Y_0$ to have the reactive flux distribution under $\Q$. In that case, the factor $\lvert \grad q(Y_0) \rvert$ in the reactive flux~\eqref{eq: reactive flux distribution} cancels with the one in~\eqref{eqn: second formula for zt}. Everything that remains in the change of measure can be computed without knowing the exact committor $q$, except for a normalizing constant. The result is formula~\eqref{eq: computable change of measure} in Theorem~\ref{thm: complete change of measure from transition path process to simulated process}.

\begin{theorem}\label{thm: complete change of measure from transition path process to simulated process}
  Let Assumptions~\ref{asm: ergodicity}-~\ref{asm: finite hitting time} hold. 
  Let $B_t$ be a $d$-dimensional standard Brownian motion on the probability space $(\Omega, \F, \Q)$. Let $Y_0$ be a random variable on $(\Omega, \F)$ that has the reactive flux distribution 
  \begin{equation*}
    Y_0 \sim  \frac{1}{\znu} \lvert  \grad q(x) \rvert \exp(-U(x)/\eps) \, \d S_A(x),
  \end{equation*}
  and assume that $Y_0$ is independent of $B_t$ under $\Q$. Let $\F_t$ be the filtration generated by $B_t$ and $Y_0$. Suppose that $Y_t$ solves the transition path equation
  \begin{equation*}
    \d Y_t = -\grad U(Y_t) \, \d t + 2 \eps \grad \log q(Y_t) \, \d t + \sqrt{2 \eps} \,  \d B_t
  \end{equation*}
  with initial condition $Y_0$. Now let $\tilde q$ be an approximation to the committor $q$, and let $m: \partial A \rightarrow [0,\infty)$ be an unnormalized density over $\partial A$ that approximates the reactive flux distribution.
    Let
  \begin{equation*}
    \zmu = \int_{\partial A} m(x) \,  \d S_A (x)
  \end{equation*}
  be the normalizing constant of $m$. 
  For $t >0$, define
  \begin{equation*}
    Z_t := \frac{\znu}{\zmu} \frac{\tilde q(Y_t)}{q(Y_t)} \frac{ m(Y_0)}{\lvert \grad \tilde q (Y_0) \rvert \exp(-U(Y_0)/\eps)}\exp \left (-\int_0^t  \frac{L \tilde q}{\tilde q}(Y_s)\, \d s \right ),
  \end{equation*}
  and let
  \begin{equation*}
    Z_0 :=  \frac{\znu}{\zmu} \frac{ m(Y_0)}{\lvert \grad q (Y_0) \rvert \exp(-U(Y_0)/\eps)}.
  \end{equation*}
  The process $Z_t$ is a nonnegative $\F_t$-martingale with $\EQ[ Z_t]=1$ for all $t \geq 0$, so there exists a probability measure $\P$ on $(\Omega, \F)$ with
  \begin{equation*}
    \left . \frac{\d \P}{\d \Q} \right \rvert_{\F_t} = Z_t
  \end{equation*}
  for all $t \geq 0$. 
  Under $\P$,
  \begin{equation*}
    W_t :=   Y_{t} - Y_0 - \int_0^{t} - \grad U (Y_s) + 2 \eps \grad \log \tilde q(Y_s) \, \d s
  \end{equation*}
  is a Brownian motion and $Y_0 \sim \zmu^{-1} m(x) \, \d S_A(x)$. 
  The density of $\P$ restricted to the stopping time $\sigma$-algebra $\F_\tau$ is
  \begin{equation}\label{eq: computable change of measure}
     \left . \frac{\d \P}{\d \Q} \right \rvert_{\F_\tau} = Z_\tau =  \frac{\znu}{\zmu} \tilde q(Y_\tau) \frac{ m(Y_0)}{\lvert \grad \tilde q (Y_0) \rvert \exp(-U(Y_0)/\eps)}\exp \left (-\int_0^\tau  \frac{L \tilde q}{\tilde q}(Y_s)\, \d s \right ).
  \end{equation}
\end{theorem}

\begin{proof}
See Appendix~\ref{apx: change of measure formula}.
\end{proof}

\section{Applications of the Change of Measure}
\label{sec: applications}

In this section, we consider two applications of the change of measure formula: \emph{selection} and \emph{training}. By training, we mean the minimization of the relative entropy $\dkl ( \P \Vert \Q)$ of the approximate transition path process with respect to the exact process over a family of approximate committor functions. By selection, we mean the comparison of different approximations to the committor function based on relative entropy differences. This is somewhat similar to model selection based on the Bayes and Akaike information criteria except that we do not consider penalizing model complexity. We also speculate on two other possibilities: \emph{error estimation} and \emph{importance sampling}. By error estimation, we mean the direct estimation of $\dkl(\P \Vert \Q)$ given a sample of trajectories from the approximate transition path process. By importance sampling, we mean the unbiased estimation of observables of the exact process given a sample of paths of the approximate process. Error estimation and importance sampling entail theoretical difficulties that we cannot resolve at this time, but see Section~\ref{sec: numerical experiments} for a simple example where both strategies appear to work well.

\subsection{Relative Entropy}
\label{sec: relative entropy in path space}

In general, if $P$ and $Q$ are probability measures on the same measurable space, one defines the \emph{relative entropy} (or Kullback--Leibler divergence) of $P$ with respect to $Q$ as follows.

\begin{definition}[Relative entropy]
  If $P$ and $Q$ are  probability measures on $(S, \mathscr{S})$ and $P$ is absolutely continuous with respect to $Q$, we define the \emph{relative entropy} of $P$ with respect to $Q$ by
  \begin{equation*}
    \dkl(P \Vert Q) :=\E_Q \left [ \frac{\d P}{\d Q} \log \left ( \frac{\d P}{\d Q} \right )  \right ] =\E_P \left [\log \left ( \frac{\d P}{\d Q} \right ) \right ]. 
  \end{equation*}
  If $P$ is not absolutely continuous with respect to $Q$, $\dkl(P \Vert Q) := \infty$. 
\end{definition}

The relative entropy is widely used in machine learning and statistics to measure discrepancies between distributions. In particular, maximum likelihood estimation can be understood as relative entropy minimization. The relative entopy is not symmetric and the triangle inequality does not hold, so it is not a metric. However, $\dkl (P \Vert Q) \geq 0$ for all $P$ and $Q$, and $\dkl (P \Vert Q)=0$ implies $P =Q$. Moreover, $\dkl (P \Vert Q)$ relates to some well-known metrics. For example, Pinsker's inequality bounds the total variation distance in terms of $\dkl$:
\begin{equation*}
  \lVert P-Q \rVert_{\mathrm{TV}}= \max_{A \in \mathscr{S}} \, \lvert P(A)- Q(A) \rvert \leq \sqrt{\frac12 \dkl(P \Vert Q)}.
\end{equation*}
%Refer to (CITE) for a detailed exposition of properties of $\dkl$. 

By abuse of notation, throughout the rest of this work, we will write $\P$ and $\Q$ for the restrictions of $\P$ and $\Q$ to the stopping time $\sigma$-algebra $\F_\tau$. In other words, we will stop observing the process $Y_t$ when it hits $B$ at time $\tau$. 
By formula~\eqref{eq: computable change of measure} in Theorem~\ref{thm: complete change of measure from transition path process to simulated process}, we have
\begin{align}
  \dkl(\P \Vert \Q) %&= \E_\Q \left [ \frac{\d \P_{\theta, \phi}}{\d \Q}  \log \left ( \frac{\d \P_{\theta, \phi}}{\d \Q} \right ) \right ] \nonumber \\
                     &= \EP \left [ \log \left ( \frac{\d \P}{\d \Q} \right ) \right ] \nonumber \\
                     &= \log \znu- \log \zmu \nonumber \\
                     &\qquad + \int_{\partial A} \log \left ( \frac{ m(x)}{\lvert \grad  \tilde q (x) \rvert \exp(-U(x)/\eps)} \right ) \frac{m(x)}{\zmu} \, \d S_A(x) \nonumber  \\
                     &\qquad + \EP \left [\log \tilde q (Y_\tau)-\int_0^\tau  \frac{L  \tilde q}{\tilde q}(Y_s)\, \d s  \right ] \label{eqn: entropy of p given q}
\end{align}
and
\begin{align}
  \dkl( \Q \Vert \P) &= \EQ \left [ \log \left ( \frac{\d \Q}{\d \P} \right ) \right ] \nonumber \\
                     &=\log \zmu - \log \znu \nonumber \\
                     &\qquad +   \int_{\partial A} \log \left ( \frac{\lvert \grad  \tilde q (x) \rvert \exp(-U(x)/\eps)}{ m(x)} \right ) \frac{\lvert \grad q(x) \rvert \exp(-U(x)/\eps)}{\znu} \, \d S_A(x) \nonumber \\
                     &\qquad + \EQ \left [  \int_0^\tau \frac{L \tilde q}{\tilde q}(Y_s) \, \d s - \log \tilde q (Y_\tau) \right ]. \label{eqn: entropy of q given p}
\end{align}
Both of these relative entropies measure the discrepancy between the exact distribution of transition paths and the distribution with an approximate committor function $\tilde q$ in place of $q$.

These formulas simplify under some conditions. First, note that $\log \tilde q(Y_\tau)=0$ when we impose $\tilde q =1$ on $\partial B$. Second, note that the integrals over $\partial A$ vanish when 
\begin{equation*}
  m(x) = \lvert \grad \tilde q(x) \rvert \exp(-U(x) / \eps).
\end{equation*}
Of course, both conditions hold for the exact committor and reactive flux distribution, and one could impose either in practice when computing approximate committor functions.   

\begin{remark}[Which Entropy?]
In the present work, we focus on $\dkl (\P \Vert \Q)$ instead of $\dkl (\Q \Vert \P)$, since $\dkl (\P \Vert \Q)$ can be estimated or minimized given a sample of approximate reactive trajectories drawn from $\P$. If one instead had a sample of exact reactive trajectories from $\Q$, then one would prefer $\dkl (\Q \Vert \P)$. A sample of exact reactive trajectories could possibly come from a very long molecular dynamics simulation of a not very metastable system or from an enhanced sampling method like parallel replica dynamics~\cite{voter_parallel_1998}. We leave it as an exercise for the reader to show that if $q_\theta$ is a smooth parametric family of functions that satisfy Assumption~\ref{asm: properties of approximate committor} and $\P_\theta$ is the distribution of the transition path process with $q_\theta$ in place of $q$, then $\dkl (\Q \Vert \P_\theta)$ is differentiable under much weaker conditions on $q_\theta$ than Assumption~\ref{asm: properties that guarantee differentiability of dkl}, which implies differentiability of $\dkl (\P_\theta \Vert \Q)$. Moreover, one can easily estimate $\grad_\theta \dkl (\Q \Vert \P_\theta)$ given a sample from $\Q$. Therefore, given a sample from $\Q$ and knowledge of the generator $L$, one could train an approximate committor to minimize $\dkl (\Q \Vert \P_\theta)$.    
\end{remark}

\begin{remark}
The existence of the change of measure $\frac{\d \P}{\d \Q}$ does not imply $\dkl( \P \Vert \Q) < \infty$. We note that $\dkl(\P \Vert \Q) < \infty$ under the same conditions that imply finite variance of the estimator in Lemma~\ref{lem: finite variance of dkl estimator}, but with $\EP[\tau^2] < \infty$ replaced by $\EP[\tau] < \infty$. We leave the proof to the reader. 
\end{remark}

\subsection{Selection: Estimating Relative Entropy Differences}
\label{sec: selection}

Let $\tilde q$ and $\bar q$ be approximate committor functions for which Assumption~\ref{asm: properties of approximate committor} holds. Let $\tilde m$ and $\bar m$ be approximate reactive flux densities supported on $\partial A$ with normalizing constants
\begin{equation*}
  \tilde \zmu := \int_{\partial A} \tilde m(x) \, \d S_A(x)  \text{ and } \bar \zmu := \int_{\partial A} \bar m(x) \, \d S_A(x).
\end{equation*}
Assume that $\tilde m$ and $\bar m$ are strictly positive. 
We define $\tilde \P$ to be the law of the approximate transition path process corresponding to $\tilde q$ and $\tilde m$, i.e.\@ the law of a weak solution of~\eqref{eq: approximate tpp} with singular drift $2\eps \grad \log \tilde q(Y_t)$  and initial condition $Y_0 \sim \tilde \zmu^{-1} m(x) \, \d S_A(x)$. We define $\bar \P$ similarly, but with $\bar q$ and $\bar m$ in place of $\tilde q$ and $\tilde m$. 
We now explain how to estimate the relative entropy difference
\begin{equation*}
  \delta(\tilde q, \tilde m ; \bar q, \bar m) := \dkl(\tilde \P \Vert \Q) - \dkl(\bar \P \Vert \Q)
\end{equation*}
given samples of paths from $\tilde \P$ and $\bar \P$. Using our estimator, one can compare the quality of approximate committor functions.
For example, one can monitor the relative entropy difference between initial and successive approximations to the committor to assess progress during training, or one can compare the quality of coarse-grained approximations to the committor depending on different numbers of variables. 
%For example, in Section~(CITE), we assess the quality of a coarse-grained approximation to the committor. In addition, one can monitor the relative entropy difference between initial and successive approximations to the committor to assess progress during training. 

Suppose that we have a sample of $\tilde N$ paths $\tilde Y^k_t$ drawn from $\tilde \P$:
\begin{equation*}
  \tilde Y^k_t \overset{\mathrm{i.i.d.}}{\sim}  \tilde \P \text{ for } k = 1, \dots, \tilde N . 
\end{equation*}
Under certain conditions on $\tilde q$, the sample average
\begin{align*}
  \tilde I_{\tilde N}:= \frac{1}{\tilde N} \sum_{k=1}^{\tilde N}   \log \left ( \frac{ m(\tilde Y^k_0)}{\lvert \grad  \tilde q (\tilde Y^k_0) \rvert \exp(-U(\tilde Y^k_0)/\eps)} \right ) + \log \tilde q(\tilde Y^k_\tau) -  \int_0^\tau \frac{L \tilde q}{\tilde q}(\tilde Y^k_s) \, \d s
\end{align*}
is a consistent and unbiased estimator of the term
\begin{align*}
  \tilde I &:= \int_{\partial A}  \log \left ( \frac{ m(x)}{\lvert \grad  \tilde q (x) \rvert \exp(-U(x)/\eps)} \right ) \frac{m(x)}{\zmu} \, \d S_A(x) \nonumber \\
                       &\qquad + \E_{\P_\theta} \left [ \log \tilde q (Y_\tau) -  \int_0^\tau \frac{L \tilde q}{\tilde q}(Y_s) \, \d s  \right ]
\end{align*}
appearing in formula~\eqref{eqn: entropy of p given q} for $\dkl ( \tilde \P \Vert \Q)$. 
 We give conditions on $\tilde q$ in Lemma~\ref{lem: finite variance of dkl estimator} in Appendix~\ref{sec: proofs relative entropy} that guarantee finite variance and a weak law of large numbers; the crucial problem is to show that the variance of $\int_0^\tau \frac{L \tilde q}{\tilde q}(\tilde Y^k_s) \, \d s$ is finite even though $\frac{L \tilde q}{\tilde q}$ would typically be singular on $\partial A$, since $\tilde q =0$ on $\partial A$. Of course, one could estimate the analogous term $\bar I$ in $\dkl (\bar \P \Vert \Q)$ by a similar average $\bar I_{\bar N}$ given a sample of $\bar N$ paths $\bar Y^k_t$ drawn from $\bar \P$. 

We have 
\begin{equation*}
 \delta(\tilde q, \tilde m ; \bar q, \bar m) = \log \left ( \frac{\bar \zmu}{\tilde \zmu} \right ) + \tilde I - \bar I,
\end{equation*}
so it remains to estimate the ratio $\frac{\bar \zmu}{\tilde \zmu}$. Analogous problems involving ratios of normalizing constants arise in Bayesian model selection and the calculation of alchemical free energy differences. Refer to~\cite{lelievre_free_2010} for a survey of methods for free energy calculations. We propose to use the Bennett Acceptance Ratio (BAR) method~\cite{bennett_efficient_1976} to estimate $ \frac{\bar \zmu}{\tilde \zmu}$. BAR is in some sense the optimal estimator of $\frac{\bar \zmu}{\tilde \zmu}$ given i.i.d.\@ samples from $\tilde m \, \d S$ and $\bar m \, \d S$. More complex alternatives that involve sampling from multiple distributions, such as the Multistate Bennett Acceptance Ratio (MBAR) method, may be needed in cases where $\tilde m$ and $\bar m$ differ greatly~\cite{shirts_statistically_2008}. 

\begin{remark}
  In practice, to sample from the approximate reactive flux densities $\tilde m$ and $\bar m$, one would most likely use a Markov chain Monte Carlo (MCMC) method. Many methods have been developed to sample distributions supported on submanifolds such as $\partial A$. For example, see~\cite[Sections 3.2.3-4]{lelievre_free_2010}. Of course, MCMC would produce a correlated sample from the approximate reactive flux distribution, not an independent sample as assumed above.  
\end{remark}

\subsection{Training On the Fly: Minimizing the Relative Entropy while Generating Reactive Trajectories}
\label{sec: training}

One could also try to minimize the relative entropy over a family of approximate committor functions by gradient descent. Let $\{q_\theta ; \theta \in \Real^k\}$ be such a family. Let each approximate committor function correspond to an approximate reactive flux distribution 
\begin{equation}\label{eq: mtheta}
  m_\theta (x) := \zmu_\theta^{-1} \grad q_\theta (x) \cdot n(x) \exp(-U(x)/\eps) \, \d S_A (x),
\end{equation}
where
\begin{equation*}
  \zmu_\theta := \int_{\partial A}  \grad q_\theta (x) \cdot n(x) \exp(-U(x)/\eps) \, \d S_A (x).
\end{equation*}
We define $\P_\theta$ to be the law of a weak solution of the stochastic differential equation~\eqref{eq: approximate tpp} observed up to time $\tau$ for the approximate committor function $q_\theta$ and initial condition $Y_0 \sim m_\theta$.
In this section, we derive a consistent estimator of $\grad_\theta \dkl ( \P_\theta \Vert \Q)$ given a sample from $\P_\theta$. We then  use the estimator in a stochastic gradient descent method to minimize $\dkl ( \P_\theta \Vert \Q)$ in Section~\ref{sec: numerical experiments}.

For our results in this section, we require much stronger assumptions on $q_\theta$ and on the distribution of $\tau$ than those imposed elsewhere; cf.\@ Assumption~\ref{asm: properties that guarantee differentiability of dkl}. The most significant of these new assumptions is that
\begin{equation}\label{eq: Lqtheta is zero on bdy A}
  L q_\theta =0 \text{ on } \partial A \text{ for all } \theta \in \Real^k.
\end{equation}
Recall that for the exact committor $q$, or more precisely for the smooth extension guaranteed by Lemma~\ref{lem: properties of committor}, we have $Lq=0$ everywhere in $\bar \D$, including on $\partial A$; cf.\@ Remark~\ref{rem: smoothness of q}.
In Section~\ref{sec: representing the committor}, we explain how to construct a practical families $q_\theta$ satisfying~\eqref{eq: Lqtheta is zero on bdy A} when the boundary of $A$ is planar. However, we do not claim that~\eqref{eq: Lqtheta is zero on bdy A} is necessary for differentiability. We simply cannot prove differentiability without imposing such an assumption. 

If $Lq_\theta=0$ on $\partial A$, then the quantity
\begin{equation*}
  \ell(x;\theta) := \frac{L q_\theta}{q_\theta}(x)
\end{equation*}
that appears in the change of measure formula~\eqref{eq: computable change of measure} extends to a smooth function of $x$ defined on an open neighborhood of $\partial A$, even though we assume $q_\theta =0$ on $\partial A$. To see this, let $Lq_\theta$ take the place of $\tilde q$ in the proof of Lemma~\ref{lem: ratio function is smooth}. Similarly, 
\begin{equation*}
  \grad_\theta \ell(x;\theta) = \frac{\grad_\theta L q_\theta (x)}{q_\theta (x)} - \frac{\grad_\theta q_\theta(x)}{q_\theta(x)} \frac{Lq_\theta(x)}{q_\theta(x)}
\end{equation*}
also extends to a smooth function on an open neighborhood of $\partial A$, since we have  $Lq_\theta= q_\theta =0$ on $\partial A$ for all $\theta$, hence $\grad_\theta Lq_\theta = \grad_\theta q_\theta =0$ on $\partial A$. We will assume that both $\ell(x;\theta)$ and $\grad_\theta \ell(x;\theta)$ are not merely smooth on a neighborhood of $\partial A$ but bounded uniformly over all $x$ and $\theta$.

Another new assumption is that, for each $\theta \in \Real^k$, there is some $\gamma(\theta) >0$ so that
\begin{equation}\label{eq: mgf of tau}
 \E_{\P_\theta}[ \exp(\gamma(\theta) \tau)] < \infty;
\end{equation}
that is, the moment generating function of $\tau$ under $\P_\theta$ is finite on an open interval containing zero. We will not verify~\eqref{eq: mgf of tau} for any particular family of approximate committor functions. However, we note that for the exact transition path process in one-dimension, under some rather stringent conditions on the potential energy $U$ and on the sets $A$ and $B$,  one can show that $\tau$ converges in distribution as $\eps \rightarrow 0$ to a shifted and scaled Gumbel random variable~\cite{cerou_length_2013}. Under much more general conditions, escape times from basins of attraction of minima of $U$ are approximately exponentially distributed for the overdamped Langevin dynamics in the limit of small $\eps$~\cite{gayrard_metastability_2004}. The moment generating function of a Gumbel or exponential random variable will be finite on an open interval containing zero, but not globally finite, which motivates~\eqref{eq: mgf of tau}. 

We summarize these new assumptions together with a few others below.

\begin{assumption}\label{asm: properties that guarantee differentiability of dkl}
  We assume that $q_\theta(x)$ is infinitely differentiable in both $x$ and $\theta$. 
  We impose the following boundary conditions: 
  \begin{itemize}
  \item $q_\theta = 0$ for $x \in \partial A$.
  \item $q_\theta (x) >0$ for $x \in \mathring{\D} \cup \partial B$. 
  \item $\grad q_\theta (x) \cdot n(x) >0$ for $x \in \partial A$. 
  \item $L q_\theta(x) =0$ for $x \in \partial A$.
  \end{itemize}
  We assume the following uniform bounds:
  \begin{itemize}
    \item
    For some $C>0$,
  \begin{equation*}
    \left \lvert \frac{L q_\theta}{q_\theta} (x) \right \rvert \leq C \text{ and } \left \lvert \grad_\theta \frac{L q_\theta}{q_\theta} (x) \right \rvert \leq C
  \end{equation*}
  for all $x \in \bar \D$ and all $\theta \in \Real^k$.
  \item For some $M >0$, 
  \begin{equation*}
    M \geq \zmu_\theta \geq \frac{1}{M} >0 \text{ and } \lvert \grad_\theta m_\theta \rvert \leq M
  \end{equation*}
  for all $\theta \in \Real^k$. 
\end{itemize}
We assume that for each $\theta \in \Real^k$, there is some $\gamma(\theta) >0$ so that
\begin{equation}
  \E_{\P_\theta}[ \exp(\gamma(\theta) \tau)] < \infty.
\end{equation}
\end{assumption}

Under Assumption~\ref{asm: properties that guarantee differentiability of dkl}, the relative entropy is differentiable. 

\begin{theorem}\label{thm: differentiability of dkl}
  Let Assumption~\ref{asm: properties that guarantee differentiability of dkl} hold, and let $\P_\theta$ and $m_\theta$ be defined as above. The relative entropy $\dkl ( \P_\theta \Vert \Q)$ is differentiable as a function of $\theta$, and 
  \begin{align*}
       \grad_\theta \dkl ( \P_\theta \Vert \Q) &= \cov_{\P_\theta} \left (\grad_\theta \log q_\theta (Y_\tau) - \int_0^\tau \grad_\theta \frac{L q_\theta}{q_\theta}(Y_s) \, \d s , \right . \\
                                              &\qquad \qquad \qquad \qquad \left .  \log q_\theta (Y_\tau) - \int_0^\tau \frac{L q_\theta}{q_\theta}(Y_s) \, \d s \right ). 
    \end{align*}
\end{theorem}

We give the proof in Appendix~\ref{sec: proofs relative entropy}. We propose to estimate $\grad_\theta \dkl( \P_\theta \Vert \Q)$ like any other covariance. Given a sample 
\begin{equation*}
   Y^k \overset{\mathrm{i.i.d.}}{\sim}   \P_\theta \text{ for } k = 1, \dots,  N,  
\end{equation*}
we define sample averages
\begin{align*}
  J_N &:= \frac{1}{N} \sum_{k=1}^N  \log q_\theta (Y^k_\tau) - \int_0^\tau \frac{L q_\theta}{q_\theta}(Y^k_s) \, \d s  \\
  K_N &:= \frac{1}{N} \sum_{k=1}^N \grad_\theta \log q_\theta (Y^k_\tau) - \int_0^\tau \grad_\theta \frac{L q_\theta}{q_\theta}(Y^k_s) \, \d s \\
\end{align*}
and the sample covariance
\begin{equation}\label{eq: covariance estimator}
  \begin{split}
G_N &:= \frac{1}{N-1} \sum_{k=1}^N  \left ( \log q_\theta (Y^k_\tau) - \int_0^\tau \frac{L q_\theta}{q_\theta}(Y^k_s) \, \d s  - J_N \right )  \\
    &\qquad \qquad \qquad \times \left ( \grad_\theta \log q_\theta (Y^k_\tau) - \int_0^\tau \grad_\theta \frac{L q_\theta}{q_\theta}(Y^k_s) \, \d s - K_N\right ).
  \end{split}
\end{equation}
Assumption~\ref{asm: properties that guarantee differentiability of dkl} implies that $G_N$ has finite variance. We leave the proof to the reader, noting that it is similar to the proof of Lemma~\ref{lem: finite variance of dkl estimator}. Using this gradient estimator, one can attempt to minimize $\dkl(\P_\theta \Vert \Q)$ by stochastic gradient descent or variants such as Adam~\cite{kingma_adam_2017}. 

\subsection{Importance Sampling and Direct Estimation of Relative Entropy}
\label{sec: importance sampling}

Using our change of measure formula, one could perhaps use importance sampling to estimate arbitrary expectations over the transition path process $\Q$ given a sample of paths from the approximate process $\P$. The discussion in this section is speculative. We do not claim that all estimators presented below are robust in practice or even that they would have finite variance under all conditions. 

Let $g : \Omega \rightarrow \Real$ be $\F_\tau$-measurable, i.e.\@ a function of the transition path process observed up to time $\tau$. By~\eqref{eq: computable change of measure}, assuming for simplicity that $\tilde q = 1$ on $\partial B$ and $m = \lvert \grad \tilde q \rvert \exp(-U /\eps)$, we have
\begin{align*}
  \EQ [ g ] &= \EP \left [ Z_\tau^{-1} g  \right ] \\
           &= \frac{\zmu}{\znu}  \EP \left [\exp \left (\int_0^{\tau(\omega)}  \frac{L \tilde q}{\tilde q}(Y_s(\omega))\, \d s \right ) g(\omega) \right ],
\end{align*}
which suggests the self-normalized importance sampling estimator
\begin{equation}\label{eqn: SNIS}
  \bar g_N = \frac{\frac{1}{N}\sum_{k = 1}^N \exp \left (\int_0^\tau  \frac{L \tilde q}{\tilde q}(Y^k_s)\, \d s \right ) g(Y^k)}{\frac{1}{N}\sum_{k = 1}^N \exp \left (\int_0^\tau  \frac{L \tilde q}{\tilde q}(Y^k_s)\, \d s \right )},
\end{equation}
of $\EQ[g]$, where $Y^k \overset{\mathrm{i.i.d.}}{\sim}   \P \text{ for } k = 1, \dots,  N$. Experience with other estimators related to Girsanov's theorem suggests that such a na\"ive estimator will not be robust. Some successes have been attained in the importance sampling of paths of diffusion processes by roughly similar means~\cite{zhang_koopman_2022,dupuis_importance_2012,vanden-eijnden_rare_2012}. We hope that our results will facilitate the development of analogous methods for the transition path process. However, we note that~\cite{zhang_koopman_2022,dupuis_importance_2012,vanden-eijnden_rare_2012} treat only paths observed up to a fixed finite time, not up to an unbounded stopping time as in our work. This is a significant challenge. In fact, since the stopping time is expected to have a moment generating function defined only a finite interval, it is not clear to us that the variance of $\bar g_N$ must be finite even if one assumes that $L \tilde q / \tilde q$ is bounded. By contrast, note that our estimators of relative entropy differences and gradients do have finite variance under plausible assumptions. Nonetheless, in Section~\ref{sec: numerical experiments}, we present numerical experiments related to a simple problem for which the self-normalized importance sampling estimator produces good results. 

One might also consider the possibility of estimating the relative entropy $\dkl(\P \Vert \Q)$, not a relative entropy difference, for a given approximate committor function. The crux of the problem is to estimate the ratio of normalizing constants
\begin{equation*}
  \frac{\znu}{\zmu}
\end{equation*}
given a sample of transition paths from $\P$. In Section~\ref{sec: selection}, we explain how to estimate all other terms comprising the relative entropy. As above, we have 
\begin{align*}
  \frac{\znu}{\zmu} &= \EP \left [\exp \left (\int_0^{\tau}  \frac{L \tilde q}{\tilde q}(Y_s)\, \d s \right )  \right ], 
\end{align*}
which suggests the estimator 
\begin{equation}\label{eq: norm constant ratio estimator}
  \frac{1}{N} \sum_{k = 1}^N \exp \left (\int_0^\tau  \frac{L \tilde q}{\tilde q}(Y^k_s)\, \d s \right )
\end{equation}
for $\znu/\zmu$, where  $Y^k \overset{\mathrm{i.i.d.}}{\sim}   \P \text{ for } k = 1, \dots,  N$. We do not expect this estimator to be any more robust than the more general importance sampling estimator above.

\section{Numerical Methods}
\label{sec: numerical methods}

In this section, we address some obstacles to the practical implementation of the estimators proposed above. In Section~\ref{sec: representing the committor}, we explain how to construct families of committor functions that satisfy Assumption~\ref{asm: properties of approximate committor} and the boundary condition $L\tilde q=0$ on $\partial A$ that arises in Section~\ref{sec: training}. In Section~\ref{sec: alternative expression}, we explain how to handle the possibly singular integrand $\frac{L \tilde q}{\tilde q}$ that appears in our change of measure formula~\eqref{eq: computable change of measure}.

\subsection{Representing the Committor Function}
\label{sec: representing the committor}
We devise a practical means of representing approximate committor functions that satisfy Assumption~\ref{asm: properties of approximate committor} and the condition
\begin{equation*}
  L \tilde q = 0 \text{ on } \partial A
\end{equation*}
appearing in Assumption~\ref{asm: properties that guarantee differentiability of dkl}.
In our approach, one must first choose a computable function $T: \D\rightarrow \Real$ with the following properties:
\begin{itemize}
\item $T$ extends to a smooth function on an open neighborhood of $\overline \D$.
\item  $\grad T(x) \cdot n(x) > 0$ for all $x \in \partial A$.
\item $T(x) =0$ for all $x \in \partial A$.
\item $T(x) >0$ for all $x \in \D \setminus \partial A$. 
\end{itemize}
For example, if $A = \{x \in \Real^d; \xi(x) \leq a\}$ were a sublevel set of some function $\xi: \Real^d \rightarrow \Real$, one could try $T=\xi-a$. Observe that all of the above conditions would hold at least locally in a neighborhood of $\partial A$ if $\lvert \grad \xi(x) \rvert$ were positive on $\partial A$. 

In our numerical experiments, the approximate committor takes the form
\begin{equation*}
  \tilde q = T \exp(w), 
\end{equation*}
where $T$ is as above and $w : \D \rightarrow \Real$ belongs to some convenient family of smooth functions. Any $\tilde q$ of this form satisfies Assumption~\ref{asm: properties of approximate committor}. In particular, the normal derivative of $\tilde q$ is positive over $\partial A$, since
\begin{align*}
  \grad \tilde q(x) \cdot n(x) &= \big \{ \grad T(x) \exp(w(x)) + \grad w(x) T(x) \exp(w(x)) \big \} \cdot n(x) \\
                               &= \grad T(x) \cdot n(x) \exp(w(x)) \\
                               &>0
\end{align*}
for $x \in \partial A$. Moreover, for any $T$ meeting the conditions above, the exact committor $q$ can be written in the form $q = T \exp(w)$ for $w = \log \left ( q/ T \right)$, and this $w$ extends to a smooth function on an open neighborhood of $\Omega$ by Lemma~\ref{lem: ratio function is smooth}.

The condition $L \tilde q = 0$ on $\partial A$ appearing in Assumption~\ref{asm: properties that guarantee differentiability of dkl} is equivalent with a Neumann boundary condition on $w$. We have
\begin{equation}\label{eq: L tilde q}
L \tilde q =  \left \{ L T + 2\eps \grad w \cdot \grad T + T \left ( L w + \eps \lvert \grad w \rvert^2 \right ) \right  \} \exp(w). 
\end{equation}
Since $T=0$ and $\grad T \cdot n >0$ on $\partial A$, $\grad T = \lvert \grad T \rvert n$ on $\partial A$, and therefore~\eqref{eq: L tilde q} implies that $L \tilde q =0$ on $\partial A$ if and only if $w$ satisfies the Neumann boundary condition
\begin{equation}\label{eq: neumann condition for lq zero}
  \grad w (x) \cdot n (x) = -\frac{1}{2 \eps} \frac{L T (x)}{\lvert \grad T (x) \rvert}
\end{equation}
for $x \in \partial A$. Of course, the difficulty of imposing such a boundary condition in practice depends on the geometry of $\partial A$.

\subsection{Alternative Expression for the Improper Integral}
\label{sec: alternative expression}
Observe that
$
  \frac{L  \tilde q}{\tilde q}
$
must be singular on $\partial A$ unless $L \tilde q =0$ on $\partial A$. Moreover, to compute $L\tilde q$ requires computing $\lap \tilde q$. \change{If $\tilde q$ is an artifical neural network and the state space is high-dimensional, then it may be inefficient to compute $\lap \tilde q$ by automatic differentiation. See~\cite{hu_hutchinson_2024} for an explanation of the problem and for a general strategy that can be used to efficiently estimate $\lap \tilde q$ in a stochastic gradient descent method at the cost of increasing the variance of the gradient estimates.}
Here, we derive a convenient alternative expression for the improper integral term
\begin{equation*}
\int_0^\tau  \frac{L  \tilde q}{\tilde q}(Y_s)\, \d s 
\end{equation*}
in the relative entropy~\eqref{eqn: entropy of p given q}. We eliminate the computationally undesirable dependence on $\lap \tilde q$. We also explain why the integral is finite even though the integrand is singular. Lemma~\ref{lem: alternative form of integral term} below is the crucial result.  

\begin{lemma}\label{lem: alternative form of integral term}
  Let $S : \D \rightarrow \Real$ be a function for which Assumption~\ref{asm: properties of approximate committor} holds. Let $w = \log \frac{\tilde q}{S}$, so
  \begin{equation*}
    \tilde q = S \exp(w).
  \end{equation*}
  We have
  \begin{align*}
    \int_0^t \frac{L\tilde{q}}{\tilde{q}}(Y_{s}) \, \d s= w(Y_{t})-w(Y_{0})&+\int_0^t\frac{LS}{S}(Y_s) -\epsilon\vert\grad w(Y_s)\vert^{2} \, \d s \\
    &- \int_0^t \sqrt{2\epsilon}\grad w(Y_s)\cdot \d W_s,
\end{align*}
where $W_t$ is the $\P$-Brownian motion in Theorem~\ref{thm: complete change of measure from transition path process to simulated process}.
\end{lemma}

We give the proof in Appendix~\ref{sec: proofs relative entropy}. For theoretical purposes, we may take $S=q$ in Lemma~\ref{lem: alternative form of integral term}, so that $w=\log r = \log (\tilde q/q)$, $LS =0$ everywhere in $\bar \D$ including on $\partial A$, and the $\frac{LS}{S}$ term vanishes. In that case, one recovers a formula for the change of measure that is similar to~\eqref{eqn: first formula for z}. We use this formula in the proofs of some results in Section~\ref{sec: applications}. For computational purposes, one can choose $S$ so that $\lap S$ is easy to compute explicitly without resorting to automatic differentiation, and then one need only compute $\grad \tilde q$, not $\lap \tilde q$, when evaluating $\int_0^\tau \frac{L\tilde{q}}{\tilde{q}}(Y_{t}) \, \d t$. Note that if $w$ is a neural network, then $\lap \tilde q$ will most likely be intractable, but $\grad \tilde q$ can be computed efficiently by backpropagation.

We also note that the formula implies that the integral $\int_0^\tau \frac{L\tilde{q}}{\tilde{q}}(Y_{t}) \, \d t$ is finite with probability one even though the integrand is singular on $\partial A$. In fact, assuming that $r$ and $\grad \log r$ are bounded, 
\begin{equation*}
\EP \left [ \int_0^h \frac{L\tilde{q}}{\tilde{q}}(Y_{s}) \, \d s \right ] = O(\sqrt{h}) \text{ and } \var_{\P} \left ( \int_0^h \frac{L\tilde{q}}{\tilde{q}}(Y_{s}) \, \d s \right ) = O(\sqrt{h})
\end{equation*}
in the limit as $h \rightarrow 0$. This suggests that it should be feasible to devise an accurate numerical integrator for the calculation of $\int_0^\tau \frac{L\tilde{q}}{\tilde{q}}(Y_{s}) \, \d s$.

% Moreover, is $L S = 0$ on $\partial A$, then the integrand $\frac{LS}{S}$ is no longer singular. We explain in Section~\ref{sec: representing the committor} above how to construct such an $S$. One possibility is
% \begin{equation*}
% S(x) = T(x) \exp \left ( - \frac{ T(x) L T( \rho(x))}{\lvert \grad T(x) \rvert^2}  \right ),
% \end{equation*}
% corresponding to taking $w_0$ and $w_2$ to be zero in~\eqref{eq: representation of w for zero Lq on boundary of A}. Here, $T$ and $\rho$ are as in Section~\ref{sec: representing the committor}. Thus, using Lemma~\ref{lem: alternative form of integral term}, for a judicious choice of $S$, one can eliminate both the singularity and the computationally intractable Laplacian of $\tilde q$. We grant, however, that for the particular $S$ defined above, it may still be difficult to compute $LS$ efficiently in practice, especially because the generator $L$ depends on $\grad U$. 

\begin{remark}
We warn the reader that since $W_t$ depends on $\tilde q$, one has to be rather careful in attempts to use this formula to simplify the calculation of derivatives of the relative entropy. 
\end{remark}

\section{Evaluating a One-Dimensional Approximation of a Toy Two-Dimensional System}
\label{sec: numerical experiments}

Here, we illustrate how the methods devised above might be used to assess the quality of a low-dimensional approximation to a committor function, to train an approximate committor function, and to compute statistics of reactive trajectories. For readability, we do not describe all aspects of the implementation of our numerical methods in this section; see Appendix~\ref{apx: numerics details} for details.

\subsection{Toy Two-Dimensional System}
We consider a two-dimensional system, where 
\begin{equation*}
  A = \{ x\in \Real^2; x_1 \leq a\} \text{ and } B = \{ x\in \Real^2; x_1 \geq b\}
\end{equation*}
for
\begin{equation*}
  a = -0.75 \text{ and } b=0.85.
\end{equation*}
We define $U_1 : \Real \rightarrow \Real$ by
\begin{align*}
  U_1(x) = 3 \left (x^2 + \frac{1}{20} \right ) \left (5 (x^2-1)^2 + \frac{x}{2} \right )
\end{align*}
and $U : \Real^2 \rightarrow \Real$ by 
\begin{equation*}
  U(x_1, x_2) = U_1(x_1) + x_2^2+ x_2(x_1-0.515)^2.
\end{equation*}
We take
\begin{equation*}
  \eps= \frac{10}{13}.
\end{equation*}
A contour plot of $U$ with the reagent and product sets indicated appears in Figure~\ref{fig: reactive-trajectory}.

\subsection{One-Dimensional Approximation to the Committor}
\label{subsec: one dim committor}
Observe that only the single term  $x_2(x_1-0.515)^2$ in $U$ couples $x_1$ with $x_2$. If one eliminates that term, then the overdamped Langevin dynamics decouples into independent equations for the coordinate components. The first component satisfies
\begin{equation*}
  \d X_{t, 1} = - U'_1(X_{t, 1} ) \, \d t + \sqrt{2 \eps } \, \d B_{t,1}.
\end{equation*}
In that case, the committor $q_1$ is a function of $x_1$ alone, and it solves
  \begin{alignat*}{2}
    -U^{\prime}_1(x_1)q_1^{\prime}(x_1)+\epsilon q_1^{\prime\prime}(x_1)&&=0, \\
    q_1(a)&&=0, \\
    q_1(b)&&=1.
\end{alignat*}
Using integrating factors yields
\begin{equation} \label{eq: exact_q_1d}
  q_1(x_1)=\frac{\int_{a}^{x_1}\exp \left ({\frac{U_1(x)}{\epsilon}} \right ) \, \d x}{\int_{a}^{b} \exp  \left ({\frac{U_1(x)}{\epsilon}}  \right ) \, \d x}.
\end{equation}
We calculated $q_1$ by quadrature using the integrate module in SciPy~\cite{2020SciPy-NMeth}; cf.\@ Figure~\ref{fig: q1}.
In the numerical experiments below, we assess the quality of $q_1$ as an approximation of the exact committor $q$ and we train an improved approximation.

\begin{figure}[h]
  \caption{The coarse-grained approximate committor $q_1$ and the corresponding effective potential $U -  2 \eps \log q_1$. In the left pane, we plot $q_1$. In the middle, we plot the effective potential $U -  2 \eps \log q_1$ corresponding to $q_1$. The gradient of the effective potential is the biasing force on reactive trajectories generated using $q_1$. For comparison, in the third panel, we plot the effective potential $U - 2\eps \log q_{\rm{fem}}$, where $ q_{\rm{fem}}$ is a very accurate finite element approximation to the committor. }
   \begin{center}
     \includegraphics{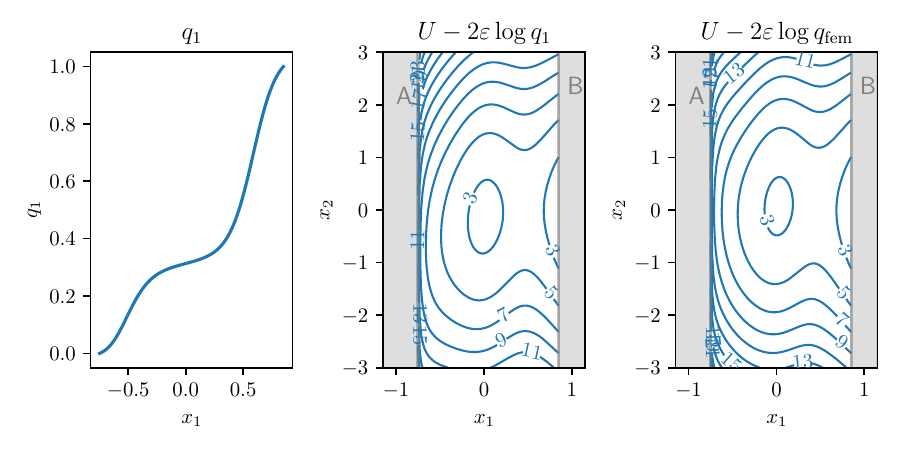}
   \end{center}
   \label{fig: q1}
\end{figure}

For comparison with $q_1$ and to validate our results, we computed a very accurate approximation of the committor by a finite element method, minimizing the Ritz functional
\begin{equation}\label{eqn: ritz form}
R(\tilde q) := \int_{[a,b] \times [-4,4]} \lvert \grad \tilde q(x) \rvert^2 \exp(-U(x)/\eps) \, \d x
\end{equation}
on the domain $[a, b] \times [4,4]$ with Dirichlet boundary conditions $q = 0$ on $\partial A$ and $q = 1$ on $\partial B$ and the natural zero Neumann boundary condition on the top and bottom of the domain where $x_2 = \pm 4$. The Neumann condition corresponds to confining the overdamped Langevin dynamics to the domain by reflection. The results appear in the middle panel of Figure~\ref{fig: err sgd multi}. Observe that the exact committor depends significantly on $x_2$.

\subsection{A Numerical Integrator Based on Operator Splitting}
\label{subsec: integrator}

Estimating the relative entropy and its gradient requires a numerical integrator for the approximate transition path process. To devise such an integrator, we write the approximate committor function in the form
\begin{equation*}
  \tilde q(x) = T(x) \exp( w(x)),
\end{equation*}
where
\begin{equation*}
  T(x) = \frac{x-a}{b-a}.
\end{equation*}
By Lemma~\ref{lem: ratio function is smooth}, any $\tilde q$ that satisfies Assumption~\ref{asm: properties of approximate committor} can be written in the above form for some smooth function $w$; cf.\@ the more general discussion in Section~\ref{sec: representing the committor}. For such $\tilde q$, the approximate transition path equation is
\begin{align}
  \d Y_t &= - \grad U(Y_t) \, \d t + 2 \eps \grad \log \tilde q (Y_t) \, \d t + \sqrt{2 \eps}\, \d B_t \nonumber \\
         &= - \grad U(Y_t) \, \d t +  2 \eps \grad w(Y_t) \, \d t + 2 \eps \grad \log T (Y_t) \, \d t+ \sqrt{2 \eps} \, \d B_t \nonumber \\
  &= - \grad U(Y_t) \, \d t +  2 \eps \grad w(Y_t) \, \d t + \frac{2 \eps}{Y_{t,1}-a} e_1 \, \d t + \sqrt{2 \eps} \, \d B_t. \label{eq: form of tpp in numerical study}
\end{align}
Here, $Y_{t,1}$ denotes the first coordinate component of the two-dimensional vector $Y_t$, and $e_1 = (1,0) \in \Real^2$ is the first standard basis vector. 
We propose an integrator based on a formal splitting of the generator of $Y_t$.

Consider the equation
\begin{equation}\label{eq: two variable bessel}
\d X_t = \frac{2 \eps }{X_{t,1} - a} e_1 \, \d t + \sqrt{2 \eps} \, \d B_t
\end{equation}
corresponding to the last two terms in~\eqref{eq: form of tpp in numerical study}. This will be the first piece in our splitting. Equation~\eqref{eq: two variable bessel} can be simulated exactly: The first coordinate component is a three-dimensional Bessel process and the second is a Brownian motion. To be precise, if $W_t$ is a three-dimensional Brownian motion, then $R_t = \lvert W_t \rvert$ is said to be a three-dimensional Bessel process, and
\begin{equation}\label{eq: bessel}
  \d R_t = \frac{1}{R_t} \d t + \d \hat W_t 
\end{equation}
for some one-dimensional Brownian motion $\hat W_t$~\cite{revuz_continuous_1999}. Equation~\eqref{eq: bessel} holds even when $R_0 = 0$. Note the similarity between the drift terms in~\eqref{eq: bessel} and~\eqref{eq: two variable bessel}. To construct a weak solution of~\eqref{eq: two variable bessel}, we may take 
\begin{equation}\label{eq: bessel construction}
X_t =(R_{2\eps t}+a, W_{2\eps t}),
\end{equation}
where $W_t$ is a one-dimensional Brownian motion that is independent of $R_t$. We observe that $X_t$ remains outside $A$ for all time, since the Bessel process $R_t$ must be nonnegative and is in fact strictly positive for $t >0$.

Increments of the Bessel-like process~\eqref{eq: two variable bessel} can be simulated conveniently. To see this, let $X_t$ solve~\eqref{eq: two variable bessel} with $X_0 = x$  where $x_1 \geq a$. Because of the relationship between $X_t$ and the three-dimensional Bessel process, $X_{\Delta t, 1}$ has the same law as 
\begin{equation*}
  \lVert (x_1-a, 0,0) + \sqrt{2 \eps \Delta t}\xi \rVert + a
\end{equation*}
where $\xi \sim N(0,I)$ is a vector of three independent standard normal random variables. We note that, in this formula, instead of $(x_1-a, 0,0)$, one could choose any vector of length $x_1 -a$, since increments of three-dimensional Brownian motion have a distribution that is symmetric under rotation.    

Our integrator alternates increments of the Bessel-like process~\eqref{eq: two variable bessel} with increments of the ordinary differential equation
\begin{equation}\label{eq: ode part of splitting}
  \d Y_t =( - \grad U(Y_t) + 2 \eps\grad w(Y_t) ) \, \d t
\end{equation}
computed by the Euler method. See Algorithm~\ref{alg: integrator} for details.  We do not prove convergence of the integrator in this work, but see Sections~\ref{sec: crossover times} and~\ref{sec: numerical entropy estimates} for numerical studies that suggest convergence.

\begin{algorithm}
  \label{alg: integrator}
  \caption{An integrator based on operator splitting.}
\begin{enumerate}
\item Fix a time step $\Delta t >0$.
\item Suppose that $\tilde Y_n$ has been calculated as an approximation of $Y_{n \Delta t}$. Generate a vector
  $$\xi_n = (\xi_{n,1}, \xi_{n,2}, \xi_{n,3}) \sim \normal (0, I)$$
   of three independent standard normal random variables.  Let $\xi_n$ be independent of all previously generated random variables. Set
   \begin{equation*}
     \tilde Y_{n+\frac12,1} = \lVert (\tilde Y_{n,1}-a, 0,0) +\sqrt{2 \eps \Delta t} \xi_n\rVert+a.
   \end{equation*}
   Let $\eta_n \sim \normal (0,1)$ be a standard normal random variable that is independent of all previously generated random variables, including $\xi_n$. Set
   \begin{equation*}
     \tilde Y_{n+\frac12,2} = \tilde Y_{n,2}+ \sqrt{2 \eps \Delta t} \eta_n.
   \end{equation*}
   Here, $\tilde Y_{n+\frac12}$ has the same law as the Bessel-like process~\eqref{eq: two variable bessel} at time $t = \Delta t$ with the initial condition $X_0= \tilde Y_n$. 
 \item Set $\tilde Y_{n+1} = \grad w(\tilde Y_{n+\frac12}) \Delta t + \tilde Y_{n+\frac12}$. That is, perform one step of the Euler method for the ordinary differential equation~\eqref{eq: ode part of splitting}. Return to step (2).
\end{enumerate}
\end{algorithm}

\subsection{Training an Improved Approximate Committor}
\label{subsec: training improved approximate committor}
 Here, we embed the low-dimensional or coarse-grained approximation $q_1$ in a parametric family of approximate committor functions and minimize the relative entropy by the stochastic gradient descent method of Section~\ref{sec: training}. We obtain both an improved approximation to the committor and also a measure of the error of the coarse-grained approximation $q_1$.
To be precise, we let
\begin{equation}\label{eq: form of q for fitting difference}
  q_\theta(x) = q_1(x) \exp( w_\theta(x) (b-x_1)),
\end{equation}
where $w_\theta$ is a tensor product of cubic B-splines on a uniform $4 \times 16$ grid covering $[a, b] \times [-3, 3]$. Every member of this family satisfies Assumption~\ref{asm: properties of approximate committor} and has $q_\theta = 1$ on $\partial B$.

We minimized $\dkl(\P_\theta \Vert \Q)$ over $\theta$, which amounts to fitting the difference $\log q - \log q_1$ by a spline using the relative entropy as a loss function. 
We used the Adam optimizer~\cite{kingma_adam_2017} and our gradient estimator~\eqref{eq: covariance estimator}.  Only two points related to the implementation are relevant to the discussion below:  We approximated the integrals in~\eqref{eq: covariance estimator} by right-handed Riemann sums over the discrete trajectories, so the first point at which $\frac{L q_\theta}{q_\theta}$ is evaluated is $X_{\Delta t}$. We took $\Delta t = 0.005$, which appears to be very close to the maximum for which the integrator is stable. See Appendix~\ref{apx: numerics details} for additional details. 

Figures~\ref{fig: err sgd multi} and~\ref{fig: convergence sgd multi} summarize the results. Figure~\ref{fig: err sgd multi} shows that the approximate committor $q_{\rm{sgd}}$ computed by gradient descent is at least qualitatively accurate. The large errors observed near the top and bottom of the domain occur in regions of extremely small probability under the Boltzmann distribution and might be considered irrelevant for most purposes. In any case, we cannot expect an accurate estimate of the committor in those regions when our reactive trajectories visit them so rarely. However, since we use only a small $4 \times 16$ grid in our spline approximation, relative errors approach $10\%$ even in some regions that reactive trajectories visit frequently. Note that relative errors, or equivalently errors in logarithms, are most relevant here, since the singular drift term $2 \eps \grad \log \tilde q$ is a function of $\log \tilde q$.

On the other hand, Figure~\ref{fig: convergence sgd multi} suggests that if one desires primarily the singular drift for simulating reactive trajectories, then $q_{\rm{sgd}}$ should be considered quantitately accurate. The relative entropy $\dkl(\P_{q_{\rm{sgd}}}\Vert \Q)$, where $\P_{q_{\rm{sgd}}}$ is the law of the transition path process with $q_{\rm{sgd}}$ in place of $q$, is approximately $0.011$; cf.\@ the right panel in Figure~\ref{fig: convergence sgd multi}. Therefore, in light of results such as Pinsker's inequality, $q_{\rm{sgd}}$ should produce reactive trajectories having nearly the correct distribution. We present a study of crossover times below to test whether this is in fact the case. We note that the Ritz form $R(q_{\rm{sgd}})$ is also very close to the minimum value. For the committor $q_{\rm{fem}}$ computed by the finite element method, $R(q_{\rm{fem}}) = 0.06612$, and $R(q_{\rm{sgd}})$ is roughly $0.06624$, cf.\@ the right panel in Figure~\ref{fig: err sgd multi}. 

\begin{figure}[h]
  \caption{Approximate committor function calculated by fitting the difference between $\log q$ and $\log q_1$ by a spline. The left pane shows the approximate committor function. The middle shows a very nearly exact committor calculated by the finite element method. The right shows the percent relative error.}
   \begin{center}
     \includegraphics{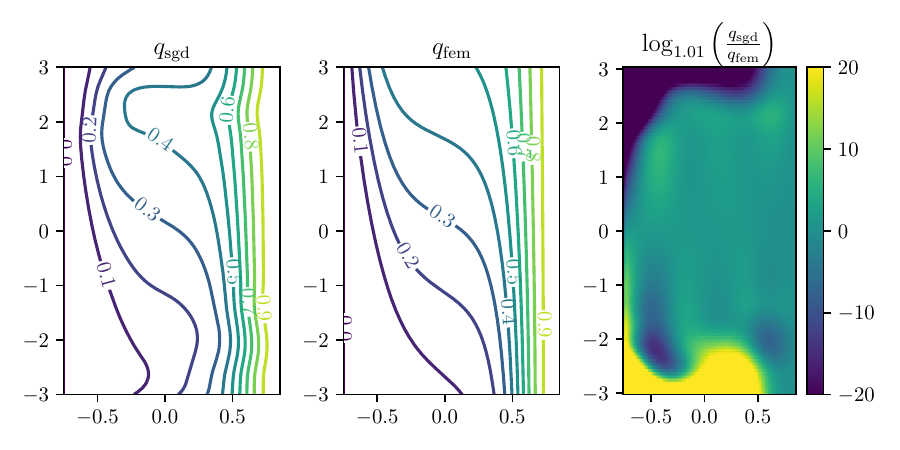}
   \end{center}
   \label{fig: err sgd multi}
\end{figure}

\begin{figure}[h] 
    \caption{Convergence of a stochastic gradient descent method for fitting the difference between $\log q$ and $\log q_1$ by a spline. The blue curve in the left pane is an estimate of $\dkl(\P_{\theta_n} \vert \Q)$ computed using~\eqref{eq: norm constant ratio estimator} with the average taken over the batch of $64$ trajectories generated in each step of gradient descent. The blue curve in the right pane is the same, but smoothed by averaging over the preceding $32$ gradient descent steps. The red curves are quadrature estimates of the Ritz form $R(q_{\theta_n})$ defined in~\eqref{eqn: ritz form}. The left pane shows the first $512$ out of $1024$ steps of gradient descent, and the right shows the last $512$. The blue scale on the vertical axis relates to the entropy estimates, and the red scale to the Ritz forms.}
  \begin{center}
    \includegraphics{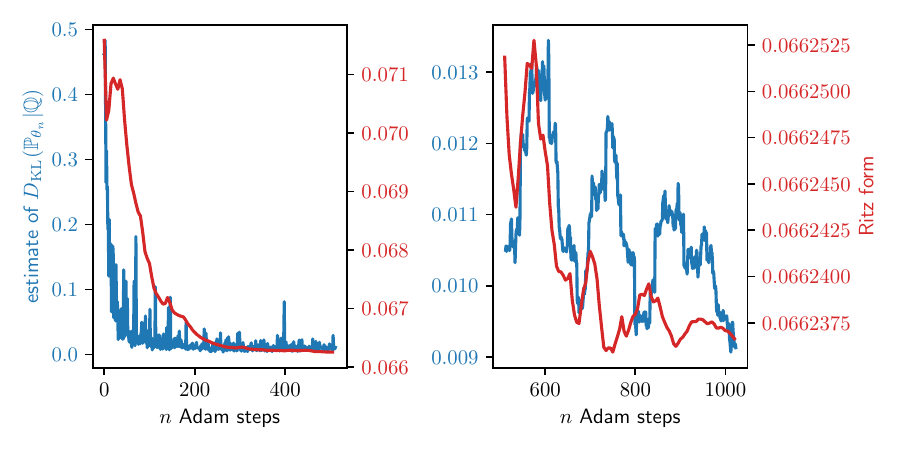}
    \end{center}
\label{fig: convergence sgd multi}
\end{figure}

\subsection{Crossover Times, Numerical Error, and Importance Sampling}
\label{sec: crossover times}
Here, we study crossover times of numerical reactive trajectories. Our primary objective is to validate the operator splitting integrator. We also demonstrate that the importance sampling formula~\eqref{eqn: SNIS} can be used to correct for errors in approximate committors to produce unbiased estimates of statistics of the exact distribution of transition paths. We calculated the very precise estimate
\begin{equation*}
\EQ[\tau]_{\rm{fem}} = 1.153
\end{equation*}
of the expected crossover time by substituting $q_{\rm{fem}}$ for $q$ in the formula 
\begin{equation}\label{eq: tpt formula for crossover time}
\EQ[\tau] = \frac{1}{\eps R(q)} \int_\D q(x) (1-q(x)) \exp(-U(x)/ \eps) \, \d x,
\end{equation}
which appears in Proposition~1.8 in~\cite{lu_reactive_2015}, estimating the integrals by quadrature. Here, $R(q)$ is the Ritz form~\eqref{eqn: ritz form}. We then generated samples of $N = 2^{15}$ approximate reactive trajectories for various values of $\Delta t$ and for the two approximate committor functions $q_1$ and $q_{\rm{sgd}}$. We computed average crossover times for each sample both with and without importance sampling. The results appear in Table~\ref{tab: expected crossover times}. With importance sampling, for either $q_1$ or $q_{\rm{sgd}}$, our integrator produces estimates of the expected crossover time that are very close to $\EQ[\tau]_{\rm{fem}}$, especially for $\Delta t \leq 0.001$. Without importance sampling, both $q_1$ and $q_{\rm{sgd}}$ produce crossover times that are too long on average, but note that the crossover times for $q_{\rm{sgd}}$ are much closer to $\EQ[\tau]_{\rm{fem}}$, indicating that training by stochastic gradient descent did result in significantly better approximate reactive trajectories, as one would expect. For the pooled sample with $\Delta t \leq 0.001$, the error in the expected crossover time for $q_{\rm{sgd}}$ is no more than $1\%$, whereas the error for $q_1$ is around $23\%$. 

\begin{table}
  \caption{Crossover times estimated by numerical integration of the transition path process with an approximate committor. All figures appear in the form $\bar \tau \pm s$, where $\bar \tau$ is the average crossover time for a sample of $N= 2^{15}$ approximate reactive trajectories and $s$ is an estimate of the standard deviation of $\bar \tau$. For importance sampling estimates, the standard deviation was computed using the delta method; cf.\@ Appendix~\ref{apx: numerics details}. Figures appear in boldface when the exact value $\EQ[\tau]_{\rm{fem}}=  1.153$ computed using the finite element method lies within one standard deviation of $\bar \tau$. The bottom row contains averages of pooled samples for the smallest three values of $\Delta t$. Note that the pooled averages computed with importance sampling are within roughly $0.5\%$ of the exact value. Here, the sample averages do not depend strongly on $\Delta t$, but compare Table~\ref{tab: model comp}.}
  \begin{center}
    \begin{tabular}{c|cccc}
\toprule
$\Delta t$ & $q_1$ without IS & $q_1$ with IS & $q_{\rm{sgd}}$ without IS & $q_{\rm{sgd}}$ with IS \\
\midrule
5.0e-03 & 1.438 $\pm$ 0.007 & 1.166 $\pm$ 0.008 & 1.182 $\pm$ 0.006 & 1.170 $\pm$ 0.006 \\
1.0e-03 & 1.406 $\pm$ 0.007 & \textbf{1.145 $\pm$ 0.008} & 1.159 $\pm$ 0.005 & \textbf{1.150 $\pm$ 0.006} \\
2.0e-04 & 1.417 $\pm$ 0.007 & \textbf{1.150 $\pm$ 0.008} & 1.167 $\pm$ 0.006 & 1.159 $\pm$ 0.006 \\
5.0e-05 & 1.415 $\pm$ 0.007 & \textbf{1.152 $\pm$ 0.008} & \textbf{1.157 $\pm$ 0.005} & \textbf{1.148 $\pm$ 0.005} \\
\midrule
 total     & 1.413 $\pm$ 0.004 & \textbf{1.149 $\pm$ 0.005} & 1.161 $\pm$ 0.003 & \textbf{1.153 $\pm$ 0.003} \\
\bottomrule
\end{tabular}
\end{center}
\label{tab: expected crossover times}
\end{table}

\subsection{Estimates of the Entropy and Entropy Differences}
\label{sec: numerical entropy estimates}
Finally, we demonstrate the most obvious way to use our results to evaluate the quality of an approximate committor function: We estimate $\dkl (\P_{q_1} \Vert \Q)$ by the method of Section~\ref{sec: importance sampling}. The results appear in Table~\ref{tab: model comp}. We used the same samples of reactive trajectories as in our study of crossover times. Observe that when estimating entropies instead of crossover times, the numerical error for large values of $\Delta t$ appears to be more significant. We suppose that this is to be expected given our na\"ive strategy of estimating the singular integral $\int_0^\tau \frac{L q_1}{q_1}(Y_s) \, \d s$ by a right-handed Riemann sum.

To validate our numerical estimates of the singular integral and related quantities such as $\dkl(\P_{q_1} \Vert \Q)$, we estimated the ratio of normalizing constants $\frac{\znu}{\zmu}$ by~\eqref{eq: norm constant ratio estimator}, we calculated $\zmu = \int_{\partial A} \lvert \grad q_1 \rvert \exp(- U/\eps) \d S$ more-or-less exactly by quadrature, and then we took the product to estimate $\znu$; cf.\@ the last column of Table~\ref{tab: model comp}. For comparison, we calculated $\znu$ by the formula $\znu = R(q)$ from Proposition~1.8 in~\cite{lu_reactive_2015} with $q_{\rm{fem}}$ in place of $q$, which yields $\znu_{\rm{fem}} = 0.06612$. Roughly speaking, errors decrease from $10\%$ to $1\%$ as the time step $\Delta
t $ decreases by a factor of $100$, which is consistent with weak convergence of order one half.

We note that numerical errors in the estimation of singular integrals do not seem to strongly affect the convergence of stochastic gradient descent. Recall that we used $\Delta t = 0.005$ when training $q_{\rm{sgd}}$, which is the maximum in Table~\ref{tab: model comp} and close to the maximum for which the integrator is stable. However, the value of the Ritz form for $q_{\rm{sgd}}$ was very nearly optimal after training: $R(q_{\rm{sgd}}) =  0.06624$ with the finite element method producing only $R(q_{\rm{fem}}) = 0.06612$, and the estimated relative entropy also decreased quite quickly during training. We tried training with smaller values of $\Delta t$, obtaining results that were only comparable, not improved, when assessed either in terms of relative entropies or Ritz forms.

As explained in Section~\ref{sec: importance sampling}, we do not expect the estimator~\eqref{eq: norm constant ratio estimator} of
$\znu / \zmu$ to be completely reliable in general since it is based on exponential averages. Therefore, for more difficult problems, if one has only a poor approximation to the committor, one might prefer to estimate entropy \emph{differences} as in Section~\ref{sec: selection}, since to compute differences does not require any exponential averages. For example, one could estimate $\dkl (\P_{q_1} \Vert \Q) - \dkl (\P_{q_{\rm{sgd}}} \Vert \Q)$ to measure how much training has improved on the initial estimate $q_1$ of the committor. For our simple toy problem, the results would be redundant with Table~\ref{tab: model comp}, since~\eqref{eq: norm constant ratio estimator} works quite well even for the relatively poor approximate committor $q_1$.      

\begin{table}
  \caption{Estimates of the entropy, singular integral, and normalizing constant based on samples of reactive trajectories biased by $q_1$. Here,  $\bar \znu$ is an estimator of the normalizing constant $\znu = R(q)$ based on sample averages of singular integrals. For comparison, $R(q_{\rm{fem}}) = 0.06612$. }
\begin{center}
\begin{tabular}{c|ccc}
\toprule
$\Delta t$ & $\dkl ( \P_{q_1} \Vert \Q )$ & $ \E_{\P_{q_1}} \left [\int_0^\tau \frac{Lq_1}{q_1}(Y_s) \, \d s \right ]$ & $\bar \znu$ \\
\midrule
5.0e-03 & 0.5970 $\pm$ 0.0066 & -1.2934 $\pm$ 0.0083 & 0.0724 $\pm$ 0.0004 \\
1.0e-03 & 0.6243 $\pm$ 0.0070 & -1.3716 $\pm$ 0.0085 & 0.0689 $\pm$ 0.0004 \\
2.0e-04 & 0.6439 $\pm$ 0.0071 & -1.4225 $\pm$ 0.0087 & 0.0667 $\pm$ 0.0003 \\
5.0e-05 & 0.6486 $\pm$ 0.0074 & -1.4309 $\pm$ 0.0086 & 0.0665 $\pm$ 0.0004 \\
\bottomrule
\end{tabular}
\end{center}
\label{tab: model comp}
\end{table}

\section{Conclusion}

We derive a change of measure formula that establishes well-posedness of the transition path equation with an approximate committor $\tilde q$ in place of the exact committor $q$. Based on the change of measure, we propose a stochastic gradient descent method that minimizes the relative entropy loss $\dkl (\P_\theta \Vert \Q)$ to train an approximate committor on the fly while computing reactive trajectories. We are interested in methods of this type for several reasons: First, techniques like adaptive multilevel splitting and PINNs require the user to specify either a reaction coordinate or collocation points in advance, which may be difficult in practice. We compute the committor, which is in many respects the ideal reaction coordinate, and instead of collocation points we generate reactive trajectories. We do require some initial approximation to the committor. We suggest that one might choose the initial approximation so that in the early stages of training the reactive trajectories amount to steered molecular dynamics trajectories with a singular, time-independent biasing force pushing on a few atoms to induce transitions to occur rapidly. Second, our loss function $\dkl (\P_\theta \Vert \Q)$ serves as a convenient and interpretable measure of the error in the distribution of reactive trajectories. It seems to us that it can be very difficult to assess the quality of solutions to high-dimensional partial differential equations obtained by PINNs and similar methods, since the usual \emph{a posteriori} error estimates for low-dimensional systems do not apply. We must be careful not to claim too much here, since estimating the normalizing constant ratio involved in $\dkl (\P_\theta \Vert \Q)$ may not always be easy, so one might prefer to work with relative entropy differences instead of relative entropies, cf.\@ Section~\ref{sec: importance sampling}. However, we obtained good results for the toy problem in Section~\ref{sec: numerical entropy estimates}. Third, in practice, neither the committor nor  a sample of reactive trajectories suffice in themselves to completely characterize rare transitions in complex systems, and one might wish to compute both at once. 

In this article, our objectives were to provide a theoretical foundation for approximations to the transition path process and to demonstrate the feasibility of computational methods for simultaneously training the committor and computing reactive trajectories. We obtained good numerical results for a two-dimensional toy problem, but to apply our proposed methods to complex problems of practical interest would entail a significant amount of additional work, including the development of a numerical integrator for the case where $A$ has a curved boundary. We also note that while the numerical studies in Sections~\ref{sec: crossover times} and~\ref{sec: numerical entropy estimates} suggest that the integrator is convergent, it appears that the estimates of the singular integrals arising in our formulas for the relative entropy and change of measure are only of weak order one half, so a higher order method might be desirable. Finally, it seems to us that a more detailed comparison of the relative entropy with other losses should be pursued, and so should a comparison of our method with other methods for the calculation of approximate committor functions. 

\appendix

\section{Proofs of Results in Section~\ref{sec: change of measure formula}}
\label{apx: change of measure formula}

\begin{proof}[Proof of Theorem~\ref{thm: complete change of measure from transition path process to simulated process}]
  Since we assume the Novikov condition, by Lemma~\ref{lem: second impractical change of measure formula}, 
  \begin{equation*}
    M_t =  \frac{\lvert \grad  q (Y_0) \rvert}{\lvert \grad \tilde q (Y_0) \rvert} \frac{\tilde q(Y_t)}{q(Y_t)} \exp \left (-\int_0^t \frac{L \tilde q}{\tilde q}(Y_s) \, \d s  \right )
  \end{equation*}
  is an $\F_t$-martingale under $\Q$ with $\EQ[M_t]=1$ for $t \geq 0$. It follows that 
  \begin{align*}
    Z_t &= \frac{\zmu^{-1} m(Y_0)}{\znu^{-1} \lvert \grad q(Y_0) \rvert \exp(-U(Y_0) /\eps)} M_t \\
   % &= \frac{\znu}{\zmu} \frac{\tilde q(Y_t)}{q(Y_t)} \frac{ m(Y_0)}{\lvert \grad \tilde q (Y_0) \rvert \exp(-U(Y_0)/\eps)}\exp \left (-\int_0^t  \frac{L \tilde q}{\tilde q}(Y_s)\, \d s \right )
  \end{align*}
  is a martingale with $\EQ[Z_t]=1$: For $0 \leq s \leq t$, we have 
  \begin{align*}
    \EQ[Z_t \vert \F_s] &= \EQ \left [  \left . \frac{\zmu^{-1} m(Y_0)}{\znu^{-1} \lvert \grad q(Y_0) \rvert \exp(-U(Y_0) /\eps)} M_t \right \vert \F_s \right ] \\
                       &=  \frac{\zmu^{-1} m(Y_0)}{\znu^{-1} \lvert \grad q(Y_0) \rvert \exp(-U(Y_0) /\eps)} \EQ [ M_t \vert \F_s] \\
                       &= \frac{\zmu^{-1} m(Y_0)}{\znu^{-1} \lvert \grad q(Y_0) \rvert \exp(-U(Y_0) /\eps)} M_s. \\
                       &= Z_s, 
  \end{align*}
  Moreover, 
  \begin{align*}
    \EQ[Z_t]&= \EQ \left [ \EQ \left [ \left .  \frac{\zmu^{-1} m(Y_0)}{\znu^{-1} \lvert \grad q(Y_0) \rvert \exp(-U(Y_0) /\eps)} M_t \right \vert \F_0 \right ] \right ] \\
           &= \EQ \left [ \frac{\zmu^{-1} m(Y_0)}{\znu^{-1} \lvert \grad q(Y_0) \rvert \exp(-U(Y_0) /\eps)} \EQ[M_t \vert \F_0] \right ] \\
           &= \EQ \left [\frac{\zmu^{-1} m(Y_0)}{\znu^{-1} \lvert \grad q(Y_0) \rvert \exp(-U(Y_0) /\eps)} M_0 \right ] \\
          &=1,
  \end{align*}
  since $M_0=1$ and $Y_0 \sim \znu^{-1} \lvert \grad q(x) \rvert \exp(-U(x)/\eps)$ under $\Q$.

  We now show that $Z_\tau$ is the density of $\P$ with respect to $\Q$ restricted to the stopping time $\sigma$-algebra $\F_\tau$. Since $Z_t$ is a martingale, the optional stopping theorem~\cite[Chapter~2, Theorem~3.2]{revuz_continuous_1999} implies that for any fixed $t >0$,
  \begin{equation*}
    Z_{t \wedge \tau} = \EQ[Z_t \vert \F_{t \wedge \tau}].
  \end{equation*}
  (Here, $t \wedge \tau$ denotes the minimium of $t$ and $\tau$.) 
  Let $A \in \F_\tau$. Since $\P[\tau < \infty]=1$, we have
  \begin{equation*}
    \P[A]=\lim_{N \rightarrow \infty} \P [A \cap \{\tau \leq N\}].
  \end{equation*}
  Now each of the sets
 \begin{equation*}
    A_N:=A\cap \{ \tau < N\} 
  \end{equation*}
  is in $\F_{\tau \wedge N} \subset \F_N$, since $A \in \F_\tau$. (Here, $\tau \wedge N$ means the minimum of $\tau$ and $N$.)
  % We have$A_N \in \F_{\tau \wedge N}$, since 
  % \begin{equation*}
  %   A_N \cap \{\tau \wedge N \leq t\} = A \cap \{\tau \leq N\} \cap \{\tau \wedge N \leq t \} = A \cap \{\tau \leq t \wedge N\} \in  \F_{t \wedge N} \subset \F_N.
  % \end{equation*}
  % Also, $\F_{\tau \wedge N} \subset \F_N$, since $\tau \wedge N \leq N$. 
  Therefore,
  \begin{align*}
    \P[A] &= \lim_{N \rightarrow \infty} \P[A_N]\\
          &\qquad \text{(since $\P[\tau < \infty ]=1$)}\\
          &= \lim_{N \rightarrow \infty}  \int_{A_N} \P (\d \omega) \\
          &=\lim_{N \rightarrow \infty} \int_{A_N} Z_N(\omega) \Q (\d \omega) \\
          &\qquad \text{(since $A_N \in \F_N$ and $Z_N = \left .  \frac{\d \P}{\d \Q} \right \rvert_{\F_N}$)} \\
          &= \lim_{N \rightarrow \infty}  \int_{A_N} \EQ[Z_N(\omega) \vert \F_{\tau(\omega) \wedge N}] \Q(\d \omega) \\
          &\qquad  \text{(since $A_N \in \F_{\tau \wedge N}$)} \\\
          &=  \lim_{N \rightarrow \infty} \int_{A_N} Z_{\tau (\omega) \wedge N}(\omega) \Q(\d \omega) \\
          &\qquad  \text{(by optional stopping, as discussed above)} \\
          &=   \lim_{N \rightarrow \infty} \int_{A_N} Z_{\tau (\omega) }(\omega) \Q(\d \omega) \\
          &\qquad  \text{(since for $\omega \in A_N$, $\tau(\omega) \wedge N = \tau(\omega)$)} \\
          &= \int_A Z_{\tau (\omega)}(\omega) \Q(\d \omega), \\
          &\qquad \text{(since $\Q[\tau < \infty ]=1$)}
  \end{align*}
  and so $Z_\tau$ is the density of $\P$ with respect to $\Q$ on $\F_\tau$. 
\end{proof}

\section{Proofs of Results in Section~\ref{sec: applications}}
\label{sec: proofs relative entropy}

First, we prove Lemma~\ref{lem: alternative form of integral term} which provides an alternative expression for the improper integral in our change of measure formula. 

\begin{proof}[Proof of Lemma~\ref{lem: alternative form of integral term}]
  By Lemma~\ref{lem: ratio function is smooth}, $w$ is smooth. Therefore, by Ito's formula, under $\P$, we have 
  \begin{equation} \label{eq: Ito_w_bar}
    dw(Y_t)=L w(Y_t) \, \d t +2\epsilon \grad \log \tilde{q}(Y_t)\cdot\grad w(Y_t) \, \d t+\sqrt{2\epsilon}\grad  w (Y_t)\cdot \d W_{t}, \\
  \end{equation}
  where $W_t$ is the $\P$-Brownian motion defined in Theorem~\ref{thm: complete change of measure from transition path process to simulated process}
  Now let
  \begin{equation*}
    \tilde{v}:=-2\epsilon\log(\tilde{q})
    \text{ and }
    v:=-2\epsilon T .
  \end{equation*}
  By a calculation similar to the proof of Lemma~\ref{lem: hjb equation}, we have  
  \begin{equation*}
    \frac{L\tilde{q}}{\tilde{q}}=-\frac{1}{2\epsilon}\left(L\tilde{v} - \frac{1}{2}\vert \grad\tilde{v} \vert^{2} \right) \text{ and }
    \frac{L T }{ T }=-\frac{1}{2\epsilon}\left(Lv - \frac{1}{2}\vert \grad v \vert^{2} \right).
  \end{equation*}
  Also, $ w =-\frac{1}{2\epsilon}(\tilde{v}- v )$. 
  Plugging these expressions into ~\eqref{eq: Ito_w_bar}, we get 
  \begin{align*}
    \d  w (Y_t)&=\left(-\frac{1}{2\epsilon}L\tilde{v}(Y_t)+\frac{1}{2\epsilon}L v (Y_t)+\frac{1}{2\epsilon}\grad\tilde{v}(Y_t)\cdot\grad(\tilde{v}- v )(Y_t) \right)\, \d t \\
               &\qquad +\sqrt{2\epsilon}\grad  w (Y_t)\cdot \d B_{t} \\
               &=-\frac{1}{2\epsilon}\left(L\tilde{v}(Y_t)-\frac{1}{2}\vert\grad\tilde{v}(Y_t)\vert^{2}\right)\d t \\
               &\qquad+\frac{1}{2\epsilon}\left(L v (Y_t)-\frac{1}{2}\vert\grad v (Y_t)\vert^{2}\right)\d t \\
               &\qquad+\frac{1}{2\epsilon}\left(\frac{1}{2}\vert\grad\tilde{v}(X_{2})\vert^{2}+\frac{1}{2}\vert\grad v (X_{2})\vert^{2}-\grad\tilde{v}(Y_t)\cdot\grad v (Y_t)\right)\d t \\
               &\qquad+\sqrt{2\epsilon}\grad  w (Y_t)\cdot \d B_{t} \\
               &=\left(\frac{L\tilde{q}}{\tilde{q}}(Y_t)-\frac{L T }{ T }(Y_t)\right)\d t \\
               &\qquad+\epsilon\vert\grad w (Y_t)\vert^{2}\d t+\sqrt{2\epsilon}\grad  w (Y_t)\cdot \d W_{t}, \\
    \end {align*}
    and the result follows.
  \end{proof}

We use Lemma~\ref{lem: alternative form of integral term} to show that our estimator of relative entropy differences has finite variance, at least under certain conditions on $\tilde q$. 
  
\begin{lemma}\label{lem: finite variance of dkl estimator}
  Assume that $\EP[\tau^2] < \infty$, that for some $C >0$,
  \begin{equation*}
    \left \lvert \log \left ( \frac{\tilde q(x)}{q(x)} \right ) \right \rvert \leq C \text{ and } \left \lvert \grad \log \left ( \frac{\tilde q(x)}{q(x)} \right ) \right \rvert \leq C
  \end{equation*}
  for all $x \in \bar \D$, and that for some $M >0$, 
  \begin{equation*}
    \frac{1}{M}\leq  \lvert m(x) \rvert \leq M \text{ and } \frac{1}{M} \leq \lvert \grad  \tilde q (x) \rvert \exp(-U(x)/\eps) \leq M
    \end{equation*}
    for all $x \in \partial A$. 
    If so,
    \begin{equation*}
      \var_\P \left (  \log \left ( \frac{ m(\tilde Y^1_0)}{\lvert \grad  \tilde q (\tilde Y^1_0) \rvert \exp(-U(\tilde Y^1_0)/\eps)} \right ) + \log \tilde q(\tilde Y^1_\tau) -  \int_0^\tau \frac{L \tilde q}{\tilde q}(\tilde Y^1_s) \, \d s \right ) < \infty,
    \end{equation*}
    so $\lim_{\tilde N \rightarrow \infty} \tilde I_{\tilde N}= \tilde I$ in $L^2(\P)$. 
\end{lemma}

\begin{proof}
  The crux of the proof is to show that
  \begin{equation*}
    \EP \left [ \left (\int_0^\tau \frac{L \tilde q}{\tilde q}(Y_s) \, \d s \right )^2 \right ] < \infty
  \end{equation*}
  even though $\frac{L \tilde q}{\tilde q}$ may be singular on $\partial A$. Taking $T=q$ in Lemma~\ref{lem: alternative form of integral term} yields
  \begin{align*}
    \int_0^\tau \frac{L \tilde q}{\tilde q}(Y_s) \, \d s &= \log \frac{r(Y_\tau)}{r(Y_0)} - \eps \int_0^\tau \lvert \grad \log r(Y_s) \rvert^2 \, \d s - \sqrt{2 \eps} \int_0^\tau \grad \log r(Y_s) \cdot \d W_s \\
    %&=: F_\tau + \eps D_\tau + \sqrt{2\eps} S_\tau,
  \end{align*}
  where $r = \tilde q/q$ and $W_s$ is the $\P$-Brownian motion in Theorem~\ref{thm: complete change of measure from transition path process to simulated process}. Therefore, 
  \begin{align}
    \EP \left [ \left (\int_0^\tau \frac{L \tilde q}{\tilde q}(Y_s) \, \d s \right )^2 \right ]^{\frac12} &\leq \EP \left [ \left ( \log \frac{r(Y_\tau)}{r(Y_0)} \right )^2 \right ]^{\frac12} + \eps \EP \left [ \left (\int_0^\tau \lvert \grad \log r(Y_s) \rvert^2 \, \d s \right )^2 \right ]^{\frac12}  \nonumber \\
                                                                                                         &\qquad \qquad \qquad +\sqrt{2 \eps} \EP \left [ \left (\int_0^\tau \grad \log r(Y_s) \cdot \d W_s \right )^2 \right ]^{\frac12} \nonumber \\
                                                                                                           &\leq 2C + \eps C \EP[\tau^2] + \sqrt{2 \eps} \EP \left [ \left (\int_0^\tau \grad \log r(Y_s) \cdot \d W_s \right )^2 \right ]^{\frac12}.  \label{eq: triangle inequality in finite variance proof}                                                                                                                         
  \end{align}  
  % We have
  % \begin{equation*}
  %    \P \left [ \left ( \log \frac{r(Y_\tau)}{r(Y_0)} \right )^2 \right ]^{\frac12} \leq  \text{ and } 
  %    \P \left [ \left ( \log \frac{r(Y_\tau)}{r(Y_0)} \right )^2 \right ]^{\frac12} \leq C \P[\tau^2].
  %  \end{equation*}
  
  We bound the third term in~\eqref{eq: triangle inequality in finite variance proof} using the Ito isometry, which entails some minor but tiresome technical difficulties, since standard forms of the isometry only apply to integrals up to finite, deterministic times. Define
  \begin{equation*}
    S_t = \int_0^t \grad \log r(Y_s) \cdot \d W_s.
  \end{equation*}
  We will approximate $S_\tau$ by $S_{\tau \wedge n}$ for $n \in \mathbb{N}$ to apply the Ito isometry; we claim that $S_{\tau \wedge n}$ is a Cauchy sequence in $L^2(\P)$. To see this, let $n< n' \in \mathbb{N}$, and observe that
  \begin{align*}
    \EP[ (S_{\tau \wedge n'} - S_{\tau \wedge n})^2 ]  &= \EP \left [ \left ( \int_{\tau \wedge n}^{\tau \wedge n'} \grad \log r(Y_s) \cdot \d W_s \right )^2 \right ]\\
                                                      &= \EP \left [ \left ( \int_0^{n'} \1_{\tau \wedge n \leq s <  \tau \wedge n'} \grad \log r(Y_s) \cdot \d W_s \right )^2 \right ]\\    &= \EP \left [  \int_0^{n'} \1_{\tau \wedge n \leq s <  \tau \wedge n'} \lvert \grad \log r(Y_s) \rvert^2 \, \d s \right ] \\
                                                      &\leq  \EP \left [  \int_0^{n'} \1_{\tau \geq n} \lvert \grad \log r(Y_s) \rvert^2 \, \d s \right ] \\
    &\leq C^2 \P[\tau \geq n].
  \end{align*}
  The third equality above follows from the Ito isometry, since $\tau$ is a stopping time, hence $\1_{\tau \wedge n \leq s <  \tau \wedge n'}$ is adapted. The first inequality holds, since $\tau < n$ implies $\1_{\tau \wedge n \leq s <  \tau \wedge n'} = 0$. We assume that $\P[\tau < \infty ] =1$, so $\lim_{n \rightarrow \infty} \P[\tau \geq n] =0$, hence $S_{\tau \wedge n}$ is Cauchy. In fact, we have $\lim_{n\rightarrow \infty} S_{\tau \wedge n} = S_\tau$ in $L^2(\P)$, since $\P[\tau < \infty]=1$ implies that $S_{\tau \wedge n}$ converges to $S_\tau$ almost surely. We conclude that
  \begin{align*}
    \EP[S_\tau^2] &= \lim_{n \rightarrow \infty} \EP[S_{\tau \wedge n}^2] \\
                 &= \lim_{n \rightarrow \infty} \EP \left [\left ( \int_0^n \1_{s \leq \tau \wedge n} \grad \log r(Y_s) \cdot \d W_s \right )^2\right ] \\
                 &=  \lim_{n \rightarrow \infty} \EP \left [ \int_0^n \1_{s \leq \tau \wedge n} \lvert \grad \log r(Y_s) \rvert^2 \, \d s\right ]\\
                 &\leq \lim_{n \rightarrow \infty} C^2 \EP[\tau \wedge n] \\
                 &= C^2 \EP[\tau]. 
  \end{align*}
  It follows that
  \begin{equation*}
      \var_\P \left (  \log \left ( \frac{ m(\tilde Y^1_0)}{\lvert \grad  \tilde q (\tilde Y^1_0) \rvert \exp(-U(\tilde Y^1_0)/\eps)} \right ) + \log \tilde q(\tilde Y^1_\tau) -  \int_0^\tau \frac{L \tilde q}{\tilde q}(\tilde Y^1_s) \, \d s \right ) < \infty.
    \end{equation*}
    We leave the remaining details to the reader. 
\end{proof}

We now prove Theorem~\ref{thm: differentiability of dkl} on the differentiability of $\dkl ( \P_\theta \Vert \Q)$ in $\theta$.

\begin{proof}[Proof of Theorem~\ref{thm: differentiability of dkl}]
  Since we assume that $\frac{L q_\theta}{q_\theta}$ and $\grad_\theta\frac{L q_\theta}{q_\theta}$ are bounded, for each $\omega \in \Omega$,
  \begin{equation*}
    \grad_\theta \int_0^{\tau(\omega)} \frac{L q_\theta}{q_\theta}(X_s(\omega)) \, \d s =  \int_0^{\tau(\omega)} \grad_\theta \frac{L q_\theta}{q_\theta}(X_s(\omega)) \, \d s . 
  \end{equation*}
  Moreover, since we assume $\zmu_\theta \geq \frac{1}{M} >0$, and since $q_\theta$ is differentiable, $\log \zmu_\theta$ is differentiable. 
  Therefore, 
  \begin{align}
    \grad_\theta  \frac{\d \P_\theta}{\d \Q} &=  \frac{\d \P_\theta}{\d \Q} \left ( \frac{\grad_\theta q_\theta (X_\tau)}{q_\theta(X_\tau)} - \frac{\grad_\theta \zmu_\theta}{\zmu_\theta} - \int_0^\tau \grad_\theta \frac{L q_\theta}{q_\theta}(X_s) \, \d s\right ). \label{eq: derivative of integrand in dkl}                                               
  \end{align}

  One now wants to interchange differentiation with respect to $\theta$ and expectation over $\Q$ to derive a formula for $\grad_\theta \dkl ( \P_\theta \Vert \Q)$. Write
  \begin{equation*}
    g(x) = x \log x
  \end{equation*}
  so that
  \begin{equation*}
    \dkl ( \P_\theta \Vert \Q ) = \EQ \left [ g \left (  \frac{\d \P_\theta}{\d \Q} \right ) \right ].
  \end{equation*}
  We claim that
  \begin{equation}\label{eq: first formula for derivative}
    \grad_\theta \dkl ( \P_\theta \Vert \Q) = \grad_\theta  \EQ \left [ g \left (  \frac{\d \P_\theta}{\d \Q} \right ) \right ] = \EQ \left [ g' \left (  \frac{\d \P_\theta}{\d \Q} \right ) \grad_\theta \frac{\d \P_\theta}{\d \Q} \right ].
  \end{equation}

  Assumption~\ref{asm: properties that guarantee differentiability of dkl} does not imply that $\grad_\theta   \frac{\d \P_\theta}{\d \Q}$ is bounded uniformly in $\theta$ by a function in $L^1(\Q)$, so textbook results related to differentiating under the integral do not apply. To verify~\eqref{eq: first formula for derivative}, we return to the definition of the derivative. Given any $h,\theta \in \Real^k$, by the mean value theorem, for some $s: \Omega \rightarrow [0,1]$, we have
  \begin{align*} 
   \frac{1}{\lVert h \rVert} \EQ &\left [ g \left (  \frac{\d \P_{\theta + h}}{\d \Q} \right )  -  g \left (  \frac{\d \P_\theta}{\d \Q} \right ) -  g' \left (  \frac{\d \P_\theta}{\d \Q} \right ) \grad_\theta \frac{\d \P_\theta}{\d \Q} \cdot h\right ] \\
    &= \EQ \left [  g' \left (  \frac{\d \P_{\theta+s h}}{\d \Q} \right ) \grad_\theta \frac{\d \P_{\theta+s h}}{\d \Q} \cdot \frac{h}{\lVert h \rVert} -  g' \left (  \frac{\d \P_\theta}{\d \Q} \right ) \grad_\theta \frac{\d \P_\theta}{\d \Q} \cdot \frac{h}{\lVert h \rVert} \right ] \\
    &= \EQ \left [ \left \{  g' \left (  \frac{\d \P_{\theta+s h}}{\d \Q} \right ) - g' \left (  \frac{\d \P_{\theta+h}}{\d \Q} \right ) \right \}  \grad_\theta \frac{\d \P_\theta}{\d \Q} \cdot \frac{h}{\lVert h \rVert} \right . \\
    &\qquad \qquad + \left . g' \left (  \frac{\d \P_{\theta+s h}}{\d \Q} \right ) \left \{\grad_\theta \frac{\d \P_{\theta+s h}}{\d \Q} - \grad_\theta \frac{\d \P_{\theta}}{\d \Q} \right \}  \cdot \frac{h}{\lVert h \rVert}\right ]\\
    &= \E_{\P_\theta} \left [ \left \{ \log \left (  \frac{\d \P_{\theta+s h}}{\d \Q} \right ) - \log \left (  \frac{\d \P_{\theta}}{\d \Q} \right ) \right \}  \grad_\theta \log \frac{\d \P_\theta}{\d \Q} \cdot \frac{h}{\lVert h \rVert} \right . \\
    &\qquad \qquad +  \left \{\log \left (  \frac{\d \P_{\theta+s h}}{\d \Q} \right ) + 1 \right \} \grad_\theta \log \frac{\d \P_{\theta+s h}}{\d \Q}  \cdot \frac{h}{\lVert h \rVert}  \frac{\d \P_{\theta+s h}}{\d \Q} \frac{\d \Q}{\d \P_\theta}   \\
                 &\qquad \qquad - \left .  \left \{\log \left (  \frac{\d \P_{\theta+s h}}{\d \Q} \right ) + 1 \right \} \grad_\theta \log \frac{\d \P_{\theta}}{\d \Q}  \cdot \frac{h}{\lVert h \rVert}\right ]\\
   &=: \E_{\P_\theta} [R(\theta; h)]. 
  \end{align*}

  As $h \rightarrow 0$, the integrand $R(\theta; h)$ in the expectation with respect to $\P_\theta$ above converges pointwise to zero.  We now show that $R(\theta; h)$ is dominated in $L^1(\P_\theta)$, uniformly in $h$ for sufficiently small $h$. If so, $\lim_{h \rightarrow 0} \E_{\P_\theta} [R(\theta; h)] =0$, which verifies differentiability of $\dkl (\P_\theta \Vert \Q)$ and~\eqref{eq: first formula for derivative}.
  By Assumption~\ref{asm: properties that guarantee differentiability of dkl}, we have
  \begin{align*}
    \left \lvert \frac{\d \P_{\theta+s h}}{\d \Q} \frac{\d \Q}{\d \P_\theta} \right \rvert &\leq \frac{\zmu_\theta}{\zmu_{\theta+sh}} \exp \left (\int_0^\tau \left \lvert \frac{Lq_\theta}{q_\theta}(Y_t) - \frac{Lq_{\theta+sh}}{q_{\theta+sh}}(Y_t) \right \rvert \, \d t \right ) \\
    &\leq M^2 \exp ( C\lvert h \rvert \tau),  
  \end{align*}
  \begin{align*}
    \left \lvert \log \left (  \frac{\d \P_{\theta}}{\d \Q} \right ) \right \rvert
    &\leq  \left \lvert \log \left ( \frac{\zmu_\theta}{\znu} \right ) \right \rvert  +  \int_0^\tau \left \lvert  \frac{L q_\theta}{q_\theta}(Y_t) \right \rvert  \, \d t \\
    &\leq \lvert \log \znu \rvert + \lvert \log M \rvert + C \tau, 
  \end{align*}
  and
  \begin{align*}
    \left \lvert \grad_\theta \log \frac{\d \P_{\theta}}{\d \Q} \right \rvert &\leq \left \lvert \frac{\grad_\theta m_\theta}{m_\theta} \right \rvert + \int_0^\tau \left \lvert \grad_\theta  \frac{L q_\theta}{q_\theta}(Y_t) \right \rvert \, \d t  \\
    &\leq M^2 + C \tau.
  \end{align*}
  Therefore, when $\lvert h \rvert \leq \frac{\gamma(\theta)}{2C}$, 
  \begin{align*}
    \lvert R(\theta; h) \rvert &\leq 2(\lvert \log \znu \rvert + \lvert \log M \rvert + C \tau)(M^2+C\tau)\\
                 &\qquad + (1+\lvert \log \znu \rvert + \lvert \log M \rvert + C \tau)(M^2+C\tau) \\
                 &\qquad + (1+\lvert \log \znu \rvert + \lvert \log M \rvert + C \tau) (M^2+C\tau)  M^2 \exp \left ( \frac{\gamma(\theta)}{2} \tau \right).
  \end{align*}
  Under Assumption~\ref{asm: properties that guarantee differentiability of dkl}, the random variable on the right-hand-side above is in $L^1(\P_\theta)$, which concludes the proof of differentiability. 
  
  We now derive a more convenient formula for the derivative.
  As a first step, we observe that one can again exchange expectation and differentiation to obtain
  \begin{align}
    0 &= \grad_\theta \E_{\P_\theta} [\Omega] \nonumber \\
      &=\grad_\theta \EQ \left [\frac{\d \P_\theta}{\d \Q} \right ] \nonumber\\
      &= \EQ \left [ \grad_\theta \frac{\d \P_\theta}{\d \Q} \right ] \nonumber\\
      &= \E_{\P_\theta} \left [  \grad_\theta \log \frac{\d \P_\theta}{\d \Q} \right ]. \label{eq: expectation of score is zero} 
      %&= \Q \left [ \frac{\d \P_\theta}{\d \Q} \left ( \frac{\grad_\theta q_\theta (X_\tau)}{q_\theta(X_\tau)} - \frac{\grad_\theta \zmu_\theta}{\zmu_\theta} - \int_0^\tau \grad_\theta \frac{L q_\theta}{q_\theta}(X_s) \, \d s\right ) \right ] \\
      %&= \Q \left [ \frac{\d \P_\theta}{\d \Q} \left ( \frac{\grad_\theta q_\theta (X_\tau)}{q_\theta(X_\tau)} - \frac{\grad_\theta \zmu_\theta}{\zmu_\theta} - \int_0^\tau \grad_\theta \frac{L q_\theta}{q_\theta}(X_s) \, \d s\right ) \right ] \\
      %&= {\P_\theta} \left [  \frac{\grad_\theta q_\theta (X_\tau)}{q_\theta(X_\tau)} - \frac{\grad_\theta \zmu_\theta}{\zmu_\theta} - \int_0^\tau \grad_\theta \frac{L q_\theta}{q_\theta}(X_s) \, \d s \right ].
  \end{align}
  Now we have 
  \begin{align*}
    \grad_\theta \dkl ( \P_\theta \vert \Q) &= \EQ \left [ g' \left ( \frac{\d \P_\theta}{\d \Q} \right ) \grad_\theta  \frac{\d \P_\theta}{\d \Q} \right ] \\
                                            &= \EQ \left [ \left ( \log \frac{\d \P_\theta}{\d \Q} + 1 \right ) \grad_\theta \frac{\d \P_\theta}{\d \Q}\right ] \\
                                            %&= \Q \left [\left ( \log \frac{\d \P_\theta}{\d \Q} \right )  \grad_\theta \frac{\d \P_\theta}{\d \Q} \right ] \\
                                            %&\qquad \text{(since $\Q\left [\grad_\theta \frac{\d \P_\theta}{\d \Q} \right ]=0$)} \\
                                            &= \E_{\P_\theta} \left [\left ( \log \frac{\d \P_\theta}{\d \Q}+1 \right )   \grad_\theta \log \frac{\d \P_\theta}{\d \Q} \right ] \\   
                                            &= \cov_{\P_\theta} \left (    \log \frac{\d \P_\theta}{\d \Q},   \grad_\theta \log \frac{\d \P_\theta}{\d \Q} \right )\\
                                            %&\qquad \text{(since $\P_\theta \left [\grad_\theta \log  \frac{\d \P_\theta}{\d \Q} \right ]=0$)} \\
                                            &= \cov_{\P_\theta} \left (     \log \frac{\znu}{\zmu_\theta} + \log q_\theta (X_\tau) - \int_0^\tau \frac{L q_\theta}{q_\theta}(X_s) \, \d s , \right .\\
                                            &\qquad \qquad \quad \left .  \frac{\grad_\theta q_\theta (X_\tau)}{q_\theta(X_\tau)} - \frac{\grad_\theta \zmu_\theta}{\zmu_\theta} - \int_0^\tau \grad_\theta \frac{L q_\theta}{q_\theta}(X_s) \, \d s \right )\\
                                            &= \cov_{\P_\theta} \left ( \log q_\theta (X_\tau) - \int_0^\tau \frac{L q_\theta}{q_\theta}(X_s) \, \d s , \right .\\
     &\qquad \qquad \quad \left .  \frac{\grad_\theta q_\theta (X_\tau)}{q_\theta(X_\tau)} - \int_0^\tau \grad_\theta \frac{L q_\theta}{q_\theta}(X_s) \, \d s \right )\\
  \end{align*}
  The fourth equality above follows from~\eqref{eq: expectation of score is zero}, and the last follows since $\log \frac{\znu}{\zmu_\theta}$ and $\frac{\grad_\theta \zmu_\theta}{\zmu_\theta}$ are constants that do not depend on $\omega$.
\end{proof}

\section{Detailed Description of Numerical Experiments}
\label{apx: numerics details}
Here, we explain the numerical experiments outlined in Section~\ref{sec: numerical experiments} in more detail.

\subsection{Details related to Section~\ref{subsec: one dim committor}}
We used FENICS to compute the committor function by a finite element method~\cite{BarattaEtal2023}. We minimized the Ritz form~\eqref{eqn: ritz form} with the prescribed boundary conditions over a P1 Lagrange finite element space based on a Delaunay triangulation of a regular rectangular grid of size $256 \times 1024$ covering the domain $[a, b] \times [-4, 4]$.

\subsection{Details related to Section~\ref{subsec: training improved approximate committor}}
When calculating an approximate committor by gradient descent, we took as the function $w_\theta (x)$ a $4 \times 16$ tensor product of cardinal cubic B-splines covering $[a,b] \times [-3,3]$. That is, for $\theta \in \Real^{4 \times 16}$ and $x \in \Real^2$, we defined
\begin{equation*}
  w_\theta (x) = \sum_{i = 0}^3 \sum_{j=0}^{15} \theta_{ij} B_3 \left ( \frac{ x_1 - a}{h_1} - i  \right ) B_3 \left ( \frac{x_2 - (-3)}{h_2} - j  \right ) ,
\end{equation*}
where
\begin{equation*}
h_1 = \frac{b-a}{4-1}, \text{ } h_2 = \frac{3-(-3)}{16-1}
\end{equation*}
and $B_3$ is the cardinal cubic B-spline
\begin{equation*}
B_3(x) := \frac{1}{6} \begin{cases}
(2 - \lvert x \rvert)^3 & \text{ if } 1 \leq \lvert x \rvert  \leq 2 \\
4 - 6x^2 - 3\lvert x \rvert^3 & \text{if } \lvert x \rvert \leq 1 \\
0 & \text{otherwise}
\end{cases}
\end{equation*}

For the Adam optimizer, we used the standard parameters listed in~\cite{kingma_adam_2017}, except that we took $\epsilon = 10^{-4}$ instead of $10^{-8}$ to mitigate Adam's tendency to make wild changes to some entries of $\theta$ when trajectories visit rare regions far from the transition tube. We took an exponentially decreasing step size schedule
\begin{equation*}
  \eta_n = 0.1^{1 + n/512}.
\end{equation*}
Although it is well known that stochastic optimization methods may not converge to a minimizer for such a rapidly decreasing schedule, we obtained better results with an exponentially decreasing schedule than with polynomially decreasing schedules for which convergence is guaranteed. Our batch size was $64$. That is, we calculated $64$ trajectories to estimate $\grad_\theta \dkl (\P_\theta \vert \Q)$ for each Adam step. 

 To simplify the sampling of initial points of trajectories, we approximated the reactive flux distribution $\tilde m$ corresponding to $\tilde q$ by a discrete distribution
\begin{equation*}
  m_d  :=  \frac{1}{\zmu_d} \sum_{i=1}^{N_d} \lvert \grad \tilde q (x_i) \rvert \exp(-U(x_i) / \eps) \delta_{x_i} 
\end{equation*}
supported on the grid 
\begin{equation*}
x_i  = \left ( a,-3 + \frac{6}{N_d -1} \right ) \text{ for } 0\leq i\leq N_d - 1
\end{equation*}
covering the interval $\{a\} \times [-3,3] \subset \partial A$ with $N_d = 1024$. 
Here,
\begin{equation*}
  \zmu_d =  \sum_{i=1}^{N_d} \lvert \grad \tilde q (x_i) \rvert \exp(-U(x_i) / \eps)
\end{equation*}
is a normalizing constant.

Finally, for efficiency, instead of computing $q_1$ defined in~\eqref{eq: exact_q_1d} by quadrature at each point of every trajectory biased by an approximate comittor of the form $q_1(x_1) \exp(w_\theta(x) (b - x_1))$ as in~\eqref{eq: form of q for fitting difference}, we interpolated a quadrature estimate of $q_1$ by a B-spline plus a linear function on a uniform grid of $16$ points covering $[a, b]$. We verified the interpolating spline to be correct to approximately five significant figures. 

\subsection{Details related to Section~\ref{sec: crossover times}}

We used the delta method to produce the standard errors in Tables~\ref{tab: expected crossover times} and~\ref{tab: model comp}. That is, to estimate the asymptotic variance of the self-normalized importance sampling estimator
\begin{equation*}
  \bar \tau := \frac{\frac{1}{N}\sum_{k = 1}^N \exp \left (\int_0^{\tau_k}  \frac{L \tilde q}{\tilde q}(Y^k_s)\, \d s \right ) \tau_k}{\frac{1}{N}\sum_{k = 1}^N \exp \left (\int_0^{\tau_k}  \frac{L \tilde q}{\tilde q}(Y^k_s)\, \d s \right )},
\end{equation*}
of the expected crossover time $\EQ[\tau]$, we first calculate the sample averages
\begin{align*}
  \tilde \tau &:= \frac{1}{N}\sum_{k = 1}^N \exp \left (\int_0^{\tau_k}  \frac{L \tilde q}{\tilde q}(Y^k_s)\, \d s \right ) \tau_k, \\
  \bar Z &:= \frac{1}{N}\sum_{k = 1}^N \exp \left (\int_0^{\tau_k}  \frac{L \tilde q}{\tilde q}(Y^k_s)\, \d s \right ) ,
\end{align*}
and the covariance matrix
\begin{equation*}
  \bar \Sigma :=
  \begin{pmatrix}
    \Sigma_{\tilde \tau, \tilde \tau} &\Sigma_{\tilde \tau, \bar Z} \\
    \Sigma_{\bar Z, \tilde \tau} &\Sigma_{\bar Z, \bar Z} \\
  \end{pmatrix},
\end{equation*}
where $\Sigma_{\tilde \tau, \bar Z}$ is the empirical covariance 
\begin{equation*}
  \Sigma_{\tilde \tau, \bar Z}  := \frac{1}{N-1} \sum_{k = 1}^N\left \{ \exp \left (\int_0^{\tau_k}  \frac{L \tilde q}{\tilde q}(Y^k_s)\, \d s \right ) \tau_k - \tilde \tau \right \} \left \{ \exp \left (\int_0^{\tau_k}  \frac{L \tilde q}{\tilde q}(Y^k_s)\, \d s \right )  - \bar Z \right \},
\end{equation*}
of $\tilde \tau$ with $\bar Z$ and the other entries of $\bar \Sigma$ are defined similarly. We then define $g(x, y) = x/y$ so that $\bar \tau = g(\tilde \tau, \bar Z)$, and the delta method~\cite{bilodeau_theory_1999} yields
\begin{equation*}
  s^2 :=  \frac{1}{N} \grad g(\tilde \tau, \bar Z)^\t \bar \Sigma \grad g( \tilde \tau, \bar Z) = \frac{1}{N} \begin{pmatrix} 1/\bar Z, & - \tilde \tau / \bar Z^2 \end{pmatrix} \bar \Sigma \begin{pmatrix} 1/\bar Z \\ - \tilde \tau / \bar Z^2 \end{pmatrix} .
\end{equation*}
as an estimate of the asymptotic variance of $\bar \tau$.

\bibliographystyle{abbrv}
\bibliography{tpp,tpp-extra}

\end{document}